\documentclass{article}
\usepackage[utf8x]{inputenc}
\usepackage{amssymb}
\usepackage{amsmath}
\usepackage{mathrsfs}
\usepackage{bm}
\usepackage{amsthm}
\usepackage{enumitem}
\usepackage{color}
\usepackage{graphicx}
\usepackage{bm}

\newcommand{\ci}[1]{\mathscr{#1}}
\newcommand{\C}{\mathbf{C}}
\newcommand{\g}[1]{\mathfrak{#1}}

\newcommand{\con}[1]{\overline{#1}}
\newcommand{\de}{\partial}

\newcommand{\alfa}{\alpha}

\newcommand{\R}{\mathbf{R}}

\newtheorem{lem}{Lemma}[section]
\newtheorem{pro}[lem]{Proposition}
\newtheorem{thm}[lem]{Theorem}
\newtheorem{rem}[lem]{Remark}

\title{On the  variation of the Einstein-Hilbert action in pseudohermitian 
geometry \\ (with an appendix by Xiaodong Wang)}

\author{Claudio Afeltra\thanks{Scuola Normale Superiore, Piazza dei Cavalieri 7, 56126 Pisa. e-mail: claudio.afeltra@sns.it}, Jih-Hsin Cheng 
	\thanks{Institute of Mathematics, Academia Sinica and NCTS 
		6F, Astronomy-Mathematics Building
		No. 1, Sec. 4
		Roosevelt Road,
		Taipei 10617, TAIWAN e-mail: cheng@math.sinica.edu.tw}, Andrea Malchiodi \thanks{Scuola Normale Superiore, Piazza dei Cavalieri 7, 56126 Pisa. e-mail: andrea.malchiodi@sns.it}, Paul Yang \thanks{Princeton University, Department of Mathematics
		Fine Hall, Washington Road, 
		Princeton NJ 08544-1000 USA e-mail: yang@math.princeton.edu}}

\begin{document}

\maketitle

\begin{abstract}
	In this paper we compute the first and second variation of the normalized Einstein-Hilbert functional 
	on  CR manifolds. We characterize critical points as pseudo-Einstein  structures. 
	We then turn to the second variation on standard spheres. While the situation is quite similar to the Riemannian case in dimension 
	greater or equal to five, in three dimension we observe a crucial difference, which mainly depends 
	on the embeddable character of  the perturbed CR structure. 
\end{abstract}

\

\section{Introduction} \label{s:intro}
 The \emph{Einstein-Hilbert action} is the functional on the space of Riemannian metrics on a closed manifold $M$ of dimension $n \geq 3$  defined by
 \begin{equation} \label{eq:EH-R}
 	\ci{R}(g)= \int_M R_gdv_g,
 \end{equation}
where $R_g$ is the scalar curvature of the metric $g$. 
 This functional is of great importance in differential geometry, and its variation is given by 
 $$ d\ci{R}(g)[h]=\int_M  E^{ij} h_{ij} dv_g$$
  where $E_{ij}$ are the components of the Einstein tensor 
 $$
 E_{ij} = R_{ij} - \frac 12 R_g g_{ij}, 
 $$
 with $R_{ij}$ the Ricci tensor of $g$.  
 This implies that the critical points of $\ci{R}$ are Ricci-flat metrics.  
 
 Its variation on asymptotically flat metrics led to the notion of {\em ADM mass}, see e.g. \cite{LeeParker}, a 
 well-studied concept in geometric relativity which has been important, among other questions, in the resolution of some 
 cases of Yamabe's conjecture, see \cite{Schoen1}. 
 
 \
 
It is also relevant to extremize the functional $\ci{R}$ constrained to the class of metrics with volume normalized 
to $1$, or equivalently to consider the scaling-invariant version 
 \begin{equation}\label{eq:tEH-R}
 	 \tilde{\ci{R}}(g) = Vol_g(M)^{\frac{2-n}{n}} \ci{R}(g). 
 \end{equation}

%
%
%
 
  The first variation of $\tilde{\ci{R}}$ is given by 
 $$
 d \tilde{\ci{R}}(g)[h] = - V(g)^{\frac{2-n}{n}} \left( \frac{n-2}{2n} V(g)^{-1} \ci{R}(g) \int_M \textrm{tr}_g h \, dv_g 
 + \int_M E^{ij} h_{ij} dv_g \right). 
 $$

 The restriction of the Einstein-Hilbert functional to a given conformal class of metrics, usually called the \emph{Yamabe functional}, plays a central role in conformal geometry (see \cite{Trudinger}, \cite{Aubin1}, \cite{Schoen1}, \cite{Aubin2}, \cite{LeeParker}).
 In fact, along a conformal variation of metric that infinitesimally preserves the volume, i.e. taking 
 $h_{ij} = \eta \, g_{ij}$ with $\int_M \eta \, dv_g = 0$, by the above formula one has 
$$
  d \tilde{\ci{R}}(g)[h] = \left( \frac{n}{2} - 1  \right) V(g)^{\frac{2-n}{n}} \int_M R_g \eta \, dv_g.  
$$ 
Criticality subject to the above constraint for $\eta$  then implies that $R_g$ is constant. 
With this condition the first variation of $\tilde{\ci{R}}$ becomes 
$$
  d \tilde{\ci{R}}(g)[h] = V(g)^{\frac{2-n}{n}} \int_M \left( \frac{R_g}{n} g^{ij} - R^{ij} \right) h_{ij} dv_g = 0, 
$$
 whose vanishing for all $h$ implies that $g$ is Einstein. 
 
 \
 
 
 The round metric of $S^n$ is a saddle point for $\tilde{\ci{R}}(\cdot)$, being a minimum restricted to its conformal class, and a local maximum in the orthogonal directions (see \cite{Schoen2}, \cite{Koiso}, \cite{FischerMarsden}, \cite{Viaclovsky}).
To see this, 
 one can split the family 
 of variations $h$ of $g_{S^n}$ as $S_0 \oplus S_1 \oplus S_2$, where $S_0$ stands for the  Lie derivatives of $g_{S^n}$,
 $S_1$ the  pure trace tensors and $S_2$ the family of TT-variations, i.e. with zero trace and divergence. The former give 
 zero contribution since they just induce a family of diffeomorphisms. 
 
 Concerning the  component $S_1$, one has that for $h = \eta g_{S^n}$  with $\eta$ 
 of zero average,  the volume of the metric is preserved at first order. The second derivative in this 
 direction is then given by 
 $$
  d^2 \tilde{\ci{R}}(g_{S^n})[h,h] = \frac{(n-1)(n-2)}{2} \int_{S^n} \left( |\nabla_{g_{S^n}} \eta|^2 - n \eta^2 \right) dv_{g_{S^n}}, 
  \quad  \int_{S^n} \eta \, dv_{g_{S^n}} = 0. 
 $$
 On the round sphere $S^n$  the first non-zero eigenvalue of the Laplacian is equal to $n$, with 
 eigenspace generated by the affine functions of $\R^{n+1}$ restricted to $S^n$, and therefore the quadratic form 
 in the latter formula is non-negative definite.

 Since the second variation of $\ci{R}$ diagonalizes with repsect to $S_0, S_1$ and $S_2$, we can therefore restrict ourselves to $h \in S_2$, 
and for this choice one finds (see e.g. Section 5.1 in \cite{Viaclovsky})
 \begin{equation}\label{eq:II-TT}
 	d^2 \tilde{\ci{R}}(g_{S^n})[h,h] =  V(g)^{\frac{2-n}{n}} \int_{S^n} \langle h, \Delta h - h \rangle dv_g, 
 \end{equation}
 where 
 $$
 (\Delta h)_{ij} = \nabla^k \nabla_k h_{ij}. 
 $$
 One sees in this way that the second variation of $\tilde{\ci{R}}$ is strictly negative-definite on the 
 subspace of variations $S_2$. 
 
 More in general, on Einstein manifolds different from the standard sphere the second 
 variation in the conformal directions is strictly positive-definite due to the 
 eigenvalue estimate in \cite{Lichnerowitz}, while on manifolds of constant curvature $K > 0$ 
 the second variation in the TT directions is strictly negative-definite, since in  \eqref{eq:II-TT} 
 the term $-h$ is replaced by $- K \, h$. 
 
 \

 The purpose of this paper is to study the situation on pseudohermitian manifolds, about which we recall the 
 definition and properties in Section \ref{s:prel}. Here we limit ourselves to mention that they are $(2n+1)$-dimensional 
 manifolds with an $n$-dimensional complex sub-bundle $\ci{H}$ of the complexified tangent bundle $T^{\C}M$ of $M$, such that $\ci{H}\cap\overline{\ci{H}} = \{0\}$, $[\ci{H},\ci{H}]\subseteq\ci{H}$, and on which a complex rotation $J$ acts. Letting 
$H(M)=\g{Re}(\ci{H}\oplus\overline{\ci{H}})$, 
 there exists 
 a one-form $\theta$ such that $H(M)=\ker\theta$.
 Classical examples are hypersurfaces of $\C^n$ or the Heisenberg group. 
 If the \emph{Levi form} $L_{\theta}(W,\con{Z}):=-id\theta(W,\con{Z})$ is non degenerate, pseudohermitian manifolds carry a natural connection, the \emph{Tanaka-Webster connection} (see \cite{DragomirTomassini}, \cite{Lee2}, \cite{Webster}), and from it curvature operators can be built in analogy with the Riemannian case. 
 The pseudohermitian counterpart of the scalar curvature is called \emph{Webster curvature}, and $\theta \wedge (d \theta)^n$ 
 acts as a natural volume form. 
 
 In the study of properties of CR geometry one observes a difference between low and high dimension related to the embeddable character of the
 underlying structures. By a classical 
 theorem by  Boutet de Monvel, see \cite{Boutet}, if a closed CR manifold $M$ of dimension $n\ge 5$ is strictly pseudoconvex (that is, the Levi form is positive-definite), then it can be CR-embedded in $\C^N$ for some natural integer $N$. This is not always true in dimension three (see \cite{ChenShaw}) and, as we will see, this fact has repercussions on the analogies between CR and Riemannian geometry. 
 
 Some conditions in three dimensions that characterize embeddability are as follows. In \cite{CCY}, Chanillo, Chiu and Yang found sufficient conditions for embeddability related to the spectral properties of the \emph{CR Paneitz operator} (see \cite{Takeuchi} for a partial converse). In \cite{CMY1}    a positive mass theorem for embeddable three-dimensional CR manifolds was proved, while, as a counterexample by the same authors in \cite{CMY2} shows, in the non embeddable case the pseudohermitian mass can be negative.
 For pseudohermitian structures which are perturbations of the standard one on $S^3$, Bland  characterized in \cite{Bland} embeddability in terms of the Fourier expansion of the deformation tensor, as it  will be recalled below. 
 
 \
 
 In analogy with \eqref{eq:EH-R} and \eqref{eq:tEH-R}, denoting by $W_{J,\theta}$ the Webster curvature on  a CR manifold $M$, we set 
 $$\ci{W}(J,\theta)= \int_{M} W_{J,\theta} \, \theta \wedge (d \theta)^n; 
 $$
 $$
  \tilde{\ci{W}}(J,\theta) = \left( \int_{M} \theta \wedge (d \theta)^n \right)^{-\frac{Q-2}{Q}} 
 \ci{W}(J,\theta), 
 $$
 where $Q = 2n+2$ stands for the {\em homogeneous dimension} of the manifold $M$. 
 
 \
 
We have first the following result.

 \begin{thm}\label{t:0}
Suppose $(J,\theta)$ is critical for $\tilde{{\ci{W}}}$. Then the Webster curvature of  
$(M,J,\theta)$ is constant and the torsion 
vanishes identically. Moreover, if $c_1(\mathcal{H}) = 0$ then $(M,J,\theta)$ is  pseudo-Einstein. 
\end{thm}

 \begin{rem}
 	Recall that a pseudohermitian structure is called \emph{pseudo-Einstein }
 	when the Ricci tensor is a constant multiple of the Levi form, see 
 	\cite{Lee2} and Section \ref{s:prel}, which is trivial when $n = 1$. We notice that by Proposition D in \cite{Lee2}, the requirement that $c_1(\mathcal{H}) = 0$ 
 	is necessary for the existence of a pseudo-Einstein structure. 
 \end{rem}
 
 The constancy of the Webster curvature and the vanishing of the torsion 
 follow from the first variation formula for $\tilde{{\ci{W}}}$. By Theorem E in \cite{Lee2}, 
 there exists a conformal choice of contact form which is pseudo-Einstein. 
 The fact that this occurs for the critical structure itself is a consequence 
 of a divergence formula by Xiaodong Wang, displayed in the appendix of the paper (see also \cite{Wang} 
 for a related result).

 We next specialize to the case of the spheres $S^{2n+1}$ endowed with the standard CR structure 
 inherited from $\C^{n+1}$, i.e.  
 $\mathcal{H}(S^{2n+1}) = T^{1,0}\C^{n+1}\cap T^{\C}_p S^{2n+1}$ and the complex 
 rotation $J_0$ is the restriction of the ambient complex one to the holomorphic 
 tangent space. A standard choice of contact form is given by 
 $$\theta_0 = \frac{i}{2}\sum_{k=1}^2 \left( z^kdz^{\con{k}}-z^{\con{k}}dz^k \right),$$
 in which case the volume element induced by  $\theta_0$ 
 compares to the Euclidean one as 
 \begin{equation*}\label{eq:vol-sphere}
 	   \theta_0 \wedge (d \theta_0)^n = 2^n n! \, dv_{\textrm{Eucl.}}. 
 \end{equation*}
This can be easily seen for example by evaluating the two volume forms at the point $(1, 0, \dots, 0) \in \R^{2n+2}$ 
and by using the homogeneity of the spherical structure.

 For $u \in C^\infty(S^{2n+1})$, we also define the {\em sub-gradient} of $u$ as 
 $$
   \nabla_b u = \Pi \nabla_{g_{S^{2n+1}}} u
 $$
where $\Pi$ denotes the orthogonal projection onto $H(S^{2n+1})$.

 In this paper we are going to exhibit another relationship between embeddability and geometric properties of CR structures.
 Starting from the three-dimensional case, we define next a variation of $J$ as 
 \begin{equation}\label{var3d}
 \dot{J}=2E = 2{E_{1}}^{\con{1}}\theta^{1}\otimes Z_{\con{1}} + \textrm{conj.},  	
 \end{equation}
 where, on $S^3$, we consider the standard generator of $\mathcal{H}$ and its dual form  
 \begin{equation}\label{eq:Z1th1}
 	Z_1=z^{\con{2}}\frac{\de}{\de z^1}-z^{\con{1}}\frac{\de}{\de z^2}, \qquad 
 	\theta^1=z^2dz^1 -z^1dz^2.
 \end{equation}



For $z_1, z_2$ the complex coordinates of $\C^2 \supseteq S^3$, we define then the subspace 
\begin{equation}\label{eq:Gamma-m}
	\Gamma_m := \left\{ f \in C^\infty(S^{3};\C) \; | \; f(e^{i \theta} z_1, e^{i \theta} z_2) = 
	e^{i m \theta} f(z_1,z_2) \right\}.
\end{equation}
For products of powers in the $z$- and $\con{z}$-coordinates, the index $m$ counts the number 
of holomorphic factors minus the anti-holomorphic ones. 

%
 
 \begin{thm}\label{t:1}
 	Consider the standard pseudohermitian structure $(S^3, J_0, \theta_0)$. 
 	\begin{itemize}
 		\item[(i)] Let  $\eta \in C^\infty(S^3)$ be such that $\int_{S^3} \eta \, \theta_0 \wedge d \theta_0 = 0$. Then 
 		$$d^2_\theta \tilde{\ci{W}}(J_0, \theta_0)[\eta \theta_0, \eta \theta_0] = c_3 
 		\int_{S^3} (|\nabla_b \eta|^2 - 3 \eta^2) \theta_0 \wedge d \theta_0 \geq 0 $$ 
 		for some positive constant $c_3$, with equality holding 
 		if and only is $\eta$ is the restriction to $S^3$ of some linear function on $\C^2$. 
 		\item[(ii)] The second derivative of $\tilde{\ci{W}}(J_0, \theta_0)$ diagonalizes 
 		with respect to the splitting in $\Gamma_m$: precisely, setting $E_1^{\con{1}} = \sum_{m \in \mathbf{Z}} E^{(m)}$, one has    
 		\begin{equation}\label{eq:dd3}
 			 		 d^2_J \tilde{\ci{W}}(J_0, \theta_0)[E, E] =  \sum_m (m+4) \int_{S^3} |E^{(m)}|^2  \theta_0 \wedge d \theta_0 
 		\end{equation}
 	\end{itemize}
 \end{thm}

Some comments on this result are in order. First, the behaviour of the functional in the conformal 
directions is completely analogous to the Riemannian case. However, a difference appears in the 
sign of second variation in the complementary directions, which we show to 
be tightly related to the embeddability properties of the (infinitesimally) perturbed CR structures. In fact, it was proved in \cite{Bland} 
via a normal form that the perturbed structures that are embeddable are precisely characterized by having vanishing 
Fourier components $E^{(m)}$ for $m \leq - 4$. Notice the difference in notation by a conjugation 
of the coefficient from (13.7) in \cite{Bland} and  \eqref{var3d}, so the present condition on $m$ 
changes by a sign compared to Theorem 15.1 in \cite{Bland}. 
We then observe a situation similar to the Riemannian 
one as long as the CR structure stays embeddable, with a reversed sign if we stay infinitesimally 
orthogonal to these and to the conformal deformations. This fact is coherent with  
the results obtained in 
\cite{CMY1} and \cite{CMY2}, where estimates on the Sobolev quotient were derived for 
a class of embeddable three-manifolds and for the (non-embeddable) Rossi spheres. 
 
 \
 
 We consider next the higher-dimensional case, which turns out to be always in analogy with the 
 Riemannian one: recall from the previous discussion that in dimensions greater or equal to five the CR structures are always 
 embeddable. 
 
 \begin{thm}\label{t:2}
 	Consider the standard pseudohermitian structure $(J_0, \theta_0)$ on $S^{2n+1}$ with $n > 1$.  
 	\begin{itemize}
 		\item[(i)] Let  $\eta \in C^\infty(S^{2n+1})$ be such that $\int_{S^{2n+1}} \eta \, \theta_0 \wedge d \theta_0 = 0$. Then 
 		$$d^2_\theta \tilde{\ci{W}}(J_0, \theta_0)[\eta \theta_0, \eta \theta_0] = 
 		c_n \int_{S^{2n+1}} (|\nabla_b \eta|^2 - n \eta^2) \theta_0 \wedge d \theta_0 \geq 0$$
 		for some positive constant $c_n$, with equality holding 
 		if and only is $\eta$ is the restriction to $S^{2n+1}$ of some linear function on $\C^{n+1}$. 
 		\item[(ii)] For the variation in $J$ we have instead 
 		$$
 		d^2_J \tilde{\ci{W}}(J_0, \theta_0)[E, E] > 0 \qquad \hbox{for any } E \not\equiv 0. 
 		$$
 	\end{itemize}
 \end{thm}

Even for $n > 1$ a formula similar to \eqref{eq:dd3} holds true, but since we would 
need to introduce extra notation, we chose to postpone it to the 
final section of the paper and state in the above theorem only its consequence. 

While in three dimensions it is a single complex function to determine a variation $E$ of the CR structure $J$, 
in higher dimensions we need to work on a vector bundle over $S^{2n+1}$. 
To carry over the calculation of first- and second-order variations we use a frame approach as in e.g. \cite{Lee2}, 
leading to formula \eqref{eq:ddot-intR} for the second variation of the integral of the  Webster curvature. However, to 
understand its sign it will be more practical for us to use the formalism in \cite{Geller}, that employs a basis of vector 
fields constructed from the ambient coordinates. Although this family  does not form a linearly 
independent system, it has the advantage of leading to constant-coefficient quantities, that 
are more suitable to be analyzed via Fourier modes. We then rely crucially on the results in 
\cite{BlandDuchamp1} and \cite{BlandDuchamp2}, where a parametrization of deformations of the 
CR structure of the sphere is performed via a suitable Banach manifold and via Fourier analysis.

\

The plan on the paper is as follows. In Section \ref{s:prel} we review some preliminary material  
that includes basic notions in pseudohermitian geometry: we derive in particular some 
first properties of deformations of the pseudohermitian structure. In Section \ref{s:def} we 
then derive useful formulas for the variation of interesting geometric quantities, and 
we derive in particular an expression for the second variation of the Webster curvature. 
In Section \ref{s:proofs} we prove our main theorems, with the completion of 
Theorem \ref{t:0} performed in the Appendix.

 \

 \begin{center}
 	{\bf Acknowledgments} 
 \end{center}

 	J.-H.C. (P.Y., resp.) is grateful to
 	Scuola Normale Superiore and Princeton University (Academia Sinica in Taiwan, resp.) for the
 	kind hospitality.  A.M. would like to thank
 	Academia Sinica in Taiwan and Princeton University for arranging some collaboration visits. 
 	J.-H.C. is supported by the project MOST 111-2115-M-001-005 of Ministry of Science and Technology and NCTS of Taiwan.
 	A.M. is supported by the project {\em Geometric  Problems with Loss of Compactness} from Scuola Normale Superiore.  He is also a member of GNAMPA as part of INdAM.  P. Y. acknowledges support from the NSF for the grant DMS 1509505.

 \section{Some preliminary facts} \label{s:prel}
 
 In this section we recall some useful basic material on pseudohermitian manifolds, 
 as well as some calculation concerning the variation of the contact form or of the CR structure.

 \
 
 About the forthcoming review material, we refer the reader to \cite{DragomirTomassini} or \cite{Lee2}. 
 We recall that 
 a \emph{CR manifold} is a real smooth manifold $M$ endowed with a complex sub-bundle $\ci{H}=T_{1,0}M$ of the complexified tangent bundle of $M$, $T^{\C}M$, such that $\ci{H}\cap\overline{\ci{H}} = \{0\}$
 and $[\ci{H},\ci{H}]\subseteq\ci{H}$. We will assume $M$ to be of hypersurface
 type, that is  $\dim M=2n+1$ and  $\dim_{\C}\ci{H}=n$.
 Let $H(M)$ denote the space $\g{Re}(\ci{H}\oplus\overline{\ci{H}})$. Then there exists a natural complex structure on $H(M)$ given by
 $$J(Z+\overline{Z})= i (Z-\overline{Z}).$$
 The CR structure is uniquely determined by $H(M)$ and $J$. For $H(M)$ and $J$  to generate a CR structure it is necessary that $\ci{H}$ is closed under Lie bracket operation. In three dimensions $\ci{H}$ is one-dimensional, so the condition $[\ci{H},\ci{H}]\subseteq\ci{H}$ is automatically 
 satisfied. 
 
 There exists a non-zero real differential form $\theta$ whose kernel at every point coincides with $H(M)$;
 it is unique up to scalar multiplication by a non-zero function.
 A triple $(M,J,\theta)$ as above is called a \emph{pseudohermitian structure}.
 On a pseudohermitian manifold, the Levi form is defined as 
 $$
 L_{\theta}(V,\con{W})=-id\theta(V,\overline{W})=i\theta([V,\overline{W}]).
 $$
 A CR manifold is said to be strictly \emph{pseudoconvex} (respectively, \emph{non-degenerate}) if it admits a positive definite (respectively, non-degenerate) Levi form. 
 Non-degeneracy is equivalent to the fact that $\theta$ is a contact form (see  Proposition 1.9  and formula (1.66) in \cite{DragomirTomassini}).
 In this case, there exists a unique vector field $T$ such that $i_Td\theta=0$ and $\theta(T)=1$. For example, if 
 $z_1, \dots, z_{n+1}$ are standard complex coordinates on $\C^{n+1}$, then on the unit sphere 
 $S^{2n+1}$ standard structures are given by 
 $$T= T_0 =\frac{i}{2}\sum_{\alpha = 1}^{n+1} \left( z_{\alfa}\frac{\de}{\de z_{\alfa}}-\con{z}_{\alfa}\frac{\de}{\de \con{z}_{\alfa}}\right);$$
 $$\theta= \theta_0 = \frac{i}{2}\sum_{\alpha = 1}^{n+1} \left( z_\alpha d\con{z}_\alpha-\con{z}_\alpha dz_\alpha \right).$$
 
 As mentioned in the introduction, classical examples of pseudohermitian manifolds are the Heisenberg 
 group or boundaries of pseudoconvex domains in complex spaces. 
 
 \

 On a nondegenerate pseudohermitian manifold one can introduce a connection, called the \emph{Tanaka-Webster connection}. To define it, we 
 recall some useful facts, mostly from \cite{Lee2}. If $\{T, Z_{\alpha}, Z_{\con{\alpha}}\}$ is a 
frame dual to $\{\theta, \theta^\alpha, \theta^{\con{\alpha}}\}$, we express the Levi form as 
$$
  L_\theta(f^\alpha Z_\alpha, g^{\con{\beta}} Z_{\con{\beta}}) = h_{\alpha \con{\beta}} f^\alpha g^{\con{\beta}}. 
$$
 The matrix $h_{\alpha \con{\beta}} = \delta_{\alpha}^\beta$  will be used in a standard way to raise and lower indices. The Webster 
 connection forms $\omega^\beta_\alpha$ and the torsion forms $\tau_\beta = 
 A_{\beta \alpha} \theta^\alpha$ are defined by the equations 
 \begin{equation}\label{eq:circ}
 	d \theta^\beta = \theta^\alpha \wedge \omega^\beta_\alpha + \theta \wedge \tau^\beta; 
 	\qquad \omega_{\alpha \con{\beta}} + \omega_{\con{\beta} \alpha} = d h_{\alpha \con{\beta}}; 
 	\qquad A_{\alpha \beta} = A_{\beta \alpha}. 
 \end{equation}
 Also, the \emph{curvature forms }
 $$
   \Pi^\beta_\alpha = d \omega^\beta_\alpha - \omega^\gamma_\alpha \wedge \omega^\beta_\gamma
 $$
 satisfy the \emph{structure equations }
 $$
   \Pi^\beta_\alpha = R^{\beta}_{\alpha \rho \con{\sigma}} \theta^\rho \wedge \theta^{\con{\sigma}} 
   + W^\beta_{\alpha \gamma} \theta^\gamma \wedge \theta - 
   W^\beta_{\alpha \con{\gamma}} \theta^{\con{\gamma}} \wedge \theta + i \theta_\alpha \wedge \tau^\beta 
   - i \tau_\alpha \wedge \theta^\beta. 
 $$
 The \emph{Ricci tensor} and the pseudohermitian scalar (or \emph{Webster}) curvature are defined by the 
 contractions 
 $$
   R_{\rho \con{\sigma}} = R^\alpha_{\alpha \rho \con{\sigma}}; \qquad \quad W = R^\alpha_\alpha. 
 $$
Recall that $(M,J,\theta)$ is {\em pseudo-Einstein} when the Ricci tensor is a 
scalar multiple of the Levi metric. A consequence of the Bianchi identities from \cite{Lee2} is 
the constancy of the Webster curvature for pseudo-Einstein structures with vanishing torsion.

\

The covariant differentiation is characterized by the formulas 
\begin{equation}\label{eq:der-cov}
	 \nabla Z_\alpha = \omega^\beta_\alpha \otimes Z_\beta; \qquad 
	\nabla Z_{\con{\alpha}} = \omega^{\con{\beta}}_{\con{\alpha}} \otimes Z_{\con{\beta}}; \qquad 
	\nabla \, T = 0. 
\end{equation}
For a tensor $S$ with components $S^\cdot_\cdot$ we will use the notation
$$
  S_{\cdot, \alpha}^{\cdot} := (\nabla_{Z_{\alpha}} S)_{\cdot}^{\cdot}; \qquad 
     S_{\cdot, \con{\alpha}}^{\cdot} := (\nabla_{Z_{\con{\alpha}}} S)_{\cdot}^{\cdot}. 
$$
We also have the following commutation relations for second-order covariant derivatives 
of functions $u$ and $(1,0)$-forms $\sigma = \sigma_\alpha \theta^\alpha$, see e.g. Lemma 2.3 in \cite{Lee2}: 
$$
u_{,\alpha \con{\beta}} - u_{,\con{\beta} \alpha} = i h_{\alpha \con{\beta}} u_{,0}; \qquad 
u_{,\alpha \beta} = u_{,\beta \alpha}; \qquad u_{,0 \alpha} - u_{,\alpha 0} = A_{\alpha \beta} u^{,\beta}; 
$$
$$
 \sigma_{\alpha, \beta \gamma} - \sigma_{\alpha,\gamma \beta} = i A_{\alpha \gamma} \sigma_\beta 
 - i A_{\alpha \beta} \sigma_\gamma; 
 $$
 $$ 
 \sigma_{\alpha, \con{\beta} \con{\gamma}} - \sigma_{\alpha,\con{\gamma} \con{\beta}} = i
 h_{\alpha \con{\beta}} A_{\con{\gamma} \con{\rho}} \sigma^{\con{\rho}} 
 - i h_{\alpha \con{\gamma}} A_{\con{\beta} \con{\rho}} \sigma^{\con{\rho}}; 
$$
$$
\sigma_{\alpha,\beta \con{\gamma}} - \sigma_{\alpha, \con{\gamma} \beta} = i h_{\beta \con{\gamma}} 
\sigma_{\alpha,0} + R^\rho_{\alpha\beta \con{\gamma}} \sigma_\rho;  
$$
$$
 \sigma_{\alpha,0 {\beta}} - \sigma_{\alpha, \beta 0} = \sigma_{\alpha ,}^{{\gamma}} A_{\gamma \beta} - 
 \sigma_\gamma A_{\alpha \beta,}^\gamma; \qquad 
 \sigma_{\alpha, 0 \beta} - \sigma_{\alpha, \beta 0} = \sigma_{\alpha, \gamma} A^\gamma_\beta 
 + \sigma_\gamma A^\gamma_{\beta, \alpha}.  
$$
As a consequence of Bianchi's identities, see e.g. Lemma 2.2 in \cite{Lee2}, we have in 
particular that 
\begin{equation}\label{eq:Bianchi}
	A_{\alpha \beta, \gamma} = A_{\alpha \gamma, \beta} \quad \hbox{ for all indices } 
 \alpha, \beta, \gamma. 
\end{equation}

\

The sub-Laplacian of a scalar function $u \in C^\infty(M)$ is then defined as (see e.g. formula (4.10) in \cite{Lee1})
$$
  \Delta_b u = u^\alpha_{, \alpha} + u_{, \con{\alpha}}^{\con{\alpha}}. 
$$
The following result will be useful to derive tensorial identities because of the 
vanishing of some terms in their expressions.  

\begin{lem} (\cite{Lee2}) \label{l:omegazero}
	For any point $p \in M$, there exists a neighborhood $U_p$ and an admissible 
	co-frame $(\theta^\alpha)_\alpha$ such that $\omega^\beta_\alpha = 0$ at $p$. 
	\end{lem}

 \

We  recall next the transformation law of the Webster curvature under conformal 
changes of contact form, see e.g. \cite{JerisonLee}. If one writes $Q = 2n+2$ 
for the \emph{homogeneous dimension} of $(M,J,\theta)$, and $\tilde{\theta} = u^{\frac{4}{Q-2}} \theta$, 
then the Webster curvature $W = W_\theta$ transforms as 
\begin{equation}\label{eq:conf}
	 - \left(2 + \frac 2n \right) \Delta_b u + W_\theta u = W_{\tilde{\theta}} 
	u^{\frac{Q+2}{Q-2}}. 
\end{equation}

 It will then be  useful to recall some spectral properties of the sub-Laplace operator 
 on spheres, see e.g. \cite{Folland} or \cite{Geller}.

 \begin{pro}
  Let $H_{p,q,n}$ denote the restriction to $S^{2n+1}$ of the homogeneous 
  complex harmonic polynomials of degree $p+q$, where $p$ is the 
  holomorphic homogeneity and $q$ the antiholomorphic one. Then one has
  $$
    - \Delta_b \phi = \lambda_{p,q,n} \phi \; \; \hbox{ for all } \phi \in H_{p,q,n}; 
    \qquad \lambda_{p,q,n} = pq + \frac{n}{2} (p+q). 
  $$ 
  Moreover, let 
  $$
  \Gamma_m = \left\{ u \in C^\infty(S^{2n+1}) \; | \; u(e^{i \theta} z_1, \dots e^{i \theta} z_n)  = e^{i m \theta} u(z_1, \dots, z_n) \right\}.
  $$
  Then one has 
  $$
  T \phi = i \frac{m}{2}  \phi \quad \hbox{ for all } \quad \phi \in \Gamma_m. 
  $$ 
 \end{pro}

 \

 We next consider a deformation of the CR structure, keeping the 
 contact form fixed: we introduce some notation for this purpose and derive some preliminary properties. 
%
%
%
%

 Given the contact bundle $\xi$, consider a smooth family $t \mapsto J_{(t)}$ of 
 CR structures on $\xi$. Then, for all values of $t$, 
 $J_{(t)}:\xi\to\xi$ satisfies $J_{(t)}^2=-Id$. Take a basis of eigenvectors $(Z_{\alpha(t)})_\alpha$ such that $J_{(t)}Z_{\alpha(t)}=iZ_{\alpha(t)}$, which 
 implies for the conjugate vector fields  $J_{(t)}\con{Z}_{\alpha(t)}=-i\con{Z}_{\alpha(t)}$. In this way $J_{(t)}$ writes as 
 $$
 J_{(t)} = i\theta^{\alpha}_{(t)}\otimes Z_{\alpha(t)}-i\theta^{\con{\alpha}}_{(t)}\otimes Z_{\con{\alpha}(t)}.
 $$ 
 We begin with the following simple result. 
 
 \begin{lem} \label{l:def-E}
 	 Setting $\dot{J} := \dot{J}_{(t)} = \frac{d}{dt} J_{(t)}$,  one has the simplified relation 
 	$$
 	\dot{J}=2E = 2{E_{\alpha}}^{\con{\beta}}\theta^{\alpha}\otimes Z_{\con{\beta}} + 
 	\textrm{\emph{conj.}}
 	$$ 
 \end{lem}

 \begin{proof} By the above formula we get 
 	$\dot{J}(Z_{\alpha})= 2{E_{\alpha}}^{\con{\beta}}Z_{\con{\beta}}+2{E_{\alpha}}^{\beta}Z_{\beta}$.
 	Differentiate the relation $J_{(t)}^2=-Id$ with respect to $t$ (at $t=0$) to obtain:
 	$$\dot{J}\circ J + J\circ\dot{J}=0.$$
Expressing this relation with respect to the basis $(Z_\alpha)_\alpha$ and $({Z}_{\con{\alpha}})_\alpha$, 
we obtain 
 	$$ \left(
 	\begin{array}{cc}
 		{E_{\alpha}}^{\beta} & {E_{\alpha}}^{\con{\beta}} \\
 		{E_{\con{\alpha}}}^{\beta} & {E_{\con{\alpha}}}^{\con{\beta}}
 	\end{array} \right)
 	\left(
 	\begin{array}{cc}
 		i & 0 \\
 		0 & -i
 	\end{array} \right)
 	+
 	\left(
 	\begin{array}{cc}
 		i & 0 \\
 		0 & -i
 	\end{array} \right)
 	\left(
 	\begin{array}{cc}
 		{E_{\alpha}}^{\beta} & {E_{\alpha}}^{\con{\beta}} \\
 		{E_{\con{\alpha}}}^{\beta} & {E_{\con{\alpha}}}^{\con{\beta}}
 	\end{array} \right)
 	=0, $$
 	which implies ${E_{\alpha}}^{\beta}={E_{\con{\alpha}}}^{\con{\beta}}=0$, as claimed. 
 \end{proof}
 
We consider  next the variation of the basis vector fields with respect to $t$. 

\begin{lem} \label{l:zdot}
	Let us write the derivative of the unitary frame $(Z_\alpha)_\alpha$ and its dual forms $(\theta^\alpha)_\alpha$ as  
	$$\dot{Z}_{\alpha}= {F_{\alpha}}^{\beta}Z_{\beta} + {G_{\alpha}}^{\con{\beta}}Z_{\con{\beta}}; \qquad \quad 
	\dot{\theta}^{\con{\gamma}} = i  E^{\con{\gamma}}_l \theta^l - F^{\con{\gamma}}_{\con{l}} \theta^{\con{l}},
	$$ 
	Then we can assume that ${F_{\alpha}}^{\beta}\in\R$, 
	and that there holds 
	\begin{equation} \label{eq:star}
		{G_{\alpha}}^{\con{\beta}}=-i {E_{\alpha}}^{\con{\beta}}; \qquad 
		{F_{\alpha}}^{\beta}+{F_{\beta}}^{\alpha}=0.
	\end{equation}
	Moreover, at $t = 0$ we can take ${F_\beta}^\alpha = 0$. 
\end{lem} 
 
 \begin{proof}
 	For notational simplicity we will often write $F_\beta^\alpha$ instead of ${F_\beta}^\alpha$, 
 	$E_{\alpha}^{\con{\beta}}$ instead of ${E_{\alpha}}^{\con{\beta}}$, etc.. 
 	
 	We  have 
 	$d\theta= i\sum\theta^{\alpha}_{(t)}\wedge \theta^{\con{\alpha}}_{(t)}$ (from $h_{\alpha\con{\beta}}= \delta_{\alpha}^{\beta}$), 
 	compare also to Lemma 2.1 in \cite{ChengLee}, so
 	$$-2id\theta(Z_{\alpha(t)}\wedge Z_{\con{\beta}(t)})
 	= \delta_{\alpha}^{\beta}.$$
 	Differentiating this relation in $t$, we get 
 	\begin{align} \label{eq:circ}
 		0 & =-2i d\theta(\dot{Z}_{\alpha}\wedge Z_{\con{\beta}})-2id\theta(Z_{\alpha}\wedge \dot{Z}_{\con{\beta}}) \nonumber \\ 
 		& =-2i d\theta(({F_{\alpha}}^{\gamma}Z_{\gamma} + {G_{\alpha}}^{\con{\gamma}}Z_{\con{\gamma}})\wedge Z_{\con{\beta}})-2id\theta(Z_{\alpha}\wedge (\con{{F_{\beta}}^{\gamma}}Z_{\con{\gamma}} + \con{{G_{\beta}}^{\con{\gamma}}}Z_{\gamma}))  \nonumber
 		\\ & = -2i {F_{\alpha}}^{\beta} -2i \con{{F_{\beta}}^{\alpha}}. 
 	\end{align}

 	On the other hand, we can always compose the frame $(Z_\alpha)_\alpha$ with an element of 
 	$SU(n)$, which infinitesimally means adding to $F_\alpha^\beta$ a matrix $B_\alpha^\beta$ such that $B_\alpha^\beta = - \con{B}_\beta^\alpha$. 
 	We can choose for example to add $\con{F}_\alpha^\beta$, which satisfies this property by the above relation \eqref{eq:circ}: 
 	this means that we can take $F$ to be real, and implies the second relation in \eqref{eq:star}.

 	To get the first one, differentiate $J_{(t)}Z_{\alpha(t)}=iZ_{\alpha(t)}$ with respect to $t$, to find 
 \begin{align*}
 	 0 & = \dot{J}Z_{\alpha} +J\dot{Z}_{\alpha} -i\dot{Z}_{\alpha}  \\ 
 	 & = 2{E_{\alpha}}^{\con{\beta}}Z_{\con{\beta}} + {F_{\alpha}}^{\beta}iZ_{\beta} + {G_{\alpha}}^{\con{\beta}}(-i)Z_{\con{\beta}}-i({F_{\alpha}}^{\beta}Z_{\beta} + {G_{\alpha}}^{\con{\beta}}Z_{\con{\beta}})
 	 \\ & = 2{E_{\alpha}}^{\con{\beta}}Z_{\con{\beta}} -2i{G_{\alpha}}^{\con{\beta}}Z_{\con{\beta}}, 
 \end{align*}
 	so ${G_{\alpha}}^{\con{\beta}}=-i{E_{\alpha}}^{\con{\beta}}$, as desired. 
 	
 	To prove that we can take $F = 0$ at $t = 0$, let $\mathfrak{F} = F^\beta_\alpha$ at $t = 0$ and 
 	consider the new frame 
 	$$
 	  \tilde{Z}_{\alpha(t)} = (e^{-t\mathfrak{F}})^\beta_\alpha Z_\beta.  
 	$$
 	Then, by cancellation 
 	$$
 	\tilde{Z}_{\alpha(0)} = Z_{\alpha(0)} \quad \hbox{ and } \quad \frac{d}{dt}_{t = 0}   \tilde{Z}_{\alpha(t)}  = 
 	- i E^{\con{\beta}}_{\alpha(0)} \tilde{Z}_{\con{\beta}(0)}, 
 	$$
 	concluding the proof. 
 \end{proof}
 
 We next derive some consequences of the integrability conditions 
 \begin{equation}\label{eq:sq}
 	  \theta([Z_\alpha, Z_\beta]) = 0; \qquad \theta^{\con{\gamma}}([Z_\alpha,Z_\beta]) = 0, 
 \end{equation}
 which hold along all the deformation, i.e. for all $t$. We have the following result. 
 
\begin{lem} \label{l:E-symm}
	For all indices $\alpha, \beta, \gamma$ we have that 
	$$
	E_{\alpha \beta} = E_{\beta \alpha}; \qquad 
	E^{\con{\gamma}}_{\alpha, \beta} = E^{\con{\gamma}}_{\beta, \alpha}. 
	$$
\end{lem}

 \begin{proof}
 	We differentiate in $t$ the first relation in \eqref{eq:sq}, obtaining  
 	\begin{align*}
 		 	 \frac{d}{dt} [Z_\alpha, Z_\beta] & = [\dot{Z}_\alpha, Z_\beta] 
 		 	 + [Z_\alpha, \dot{Z}_\beta] \\ & = [F^\gamma_\alpha Z_\gamma - i E^{\con{\gamma}}_\alpha Z_{\con{\gamma}}, Z_\beta] 
 		 	 + [Z_\alpha, F^\gamma_\beta Z_\gamma - i E^{\con{\gamma}}_\beta Z_{\con{\gamma}}], 
 	\end{align*}
 which by integrability yields 
 $$
  0 =  \theta \left( \frac{d}{dt} [Z_\alpha, Z_\beta]  \right) = - i E^{\con{\gamma}}_\alpha \theta ([Z_{\con{\gamma}}, Z_\beta] )
  - i E^{\con{\gamma}}_\beta \theta ([Z_\alpha,   Z_{\con{\gamma}}]). 
 $$
 Notice that 
 $$
   [Z_{\con{\gamma}}, Z_\beta] = i h_{\beta \con{\gamma}} T + \omega^l_\beta(Z_{\con{\gamma}}) Z_l - 
   \omega^{\con{l}}_{\con{\gamma}}(Z_\beta) Z_{\con{l}},
 $$
 which in turn implies 
 $$
   0 = - i E^{\con{\gamma}}_\alpha i h_{\beta \con{\gamma}} + i E^{\con\gamma}_\beta i h_{\alpha \con{\gamma}}. 
 $$
 Since $h_{\beta \con{\gamma}} = \delta_{\beta}^{\gamma}$, we deduce  the first assertion of the lemma.
 
 \

 We next differentiate in $t$ the second relation in \eqref{eq:sq} to get 
 $$
  0 = \dot{\theta}^{\con{\gamma}}([Z_\alpha, Z_\beta]) + \theta^{\con{\gamma}} \left( \frac{d}{dt} [Z_\alpha, Z_\beta]  \right). 
 $$
 Since by Lemma \ref{l:zdot}
 $$
 \dot{\theta}^{\con{\gamma}} = i  E^{\con{\gamma}}_l \theta^l - F^{\con{\gamma}}_{\con{l}} \theta^{\con{l}},
 $$
 and 
 $$
  [Z_\alpha, Z_\beta] = \omega^l_\beta(Z_\alpha) Z_l - \omega^l_\alpha(Z_\beta) Z_l,
 $$
 we obtain 
 \begin{align*}
 	0 & =  i  E^{\con{\gamma}}_l \theta^l (\omega^l_\beta(Z_\alpha) - \omega^l_\alpha(Z_\beta)) 
 	\\ & +  \theta^{\con{\gamma}} \left( [Z_\alpha, F^\gamma_\beta Z_\gamma - i E^{\con{\gamma}}_\beta Z_{\con{\gamma}}] 
 	+ [Z_\alpha, F^\gamma_\beta Z_\gamma - i E^{\con{\gamma}}_\beta Z_{\con{\gamma}}] \right) \\ &  =
 	- i \omega^l_\alpha(Z_\beta) E^{\con{\gamma}}_l + i \omega^l_\beta(Z_\alpha) E^{\con{\gamma}}_l 
 	+ i E^{\con{l}}_\alpha \omega^{\con{\gamma}}_{\con{l}}(Z_\beta) \\ & + i Z_\beta(E^{\con{\gamma}})_\alpha
 	- i E^{\con{l}}_\beta \omega^{\con{\gamma}}_{\con{l}}(Z_\alpha) - i Z_\alpha(E^{\con{l}}_\beta).
 \end{align*}
This implies
$$
i E^{\con{\gamma}}_{\alpha, \beta} - i E^{\con{\gamma}}_{\beta, \alpha} = 0, 
$$
which is the second assertion. 
 \end{proof}

\section{Variation of geometric quantities} \label{s:def}

In this section we compute the variation of several relevant geometric quantities along the 
deformation of $J$. Our final goal is to derive a second variation formula for the 
 normalized integral of the  Webster curvature at its critical points.

We begin with the variation of connection and torsion. 

\begin{lem}
	For all $t$, the variation of the torsion is given by 
	\begin{equation} \label{eq:d-tor}
		\dot{A}^\alpha_{\con{\gamma}} = - i E^\alpha_{\con{\gamma},0} + A^\alpha_{\con{l}} F^{\con{l}}_{\con{\gamma}} 
		- F^\alpha_l A^l_{\con{\gamma}}, 
	\end{equation}
	while for the variation of the connection we have 
\begin{align} \label{eq:d-conn}
	\dot{\omega}^\alpha_\beta & = \left[ i (A^\alpha_{\con{\gamma}} E^{\con{\gamma}}_\beta + E^\alpha_{\con{\gamma}} A^{\con{\gamma}}_\beta)
	+ F^\alpha_{\beta,0} \right] 
	\theta \\ & + (- i E^{\con{\beta}}_{\gamma,\con{\alpha}} 
	- F^{\con{\beta}}_{\con{\alpha},\gamma})  \theta^\gamma + (- i E^\alpha_{\con{\gamma},\beta} + F^\alpha_{\beta,\con{\gamma}}) 
	\theta^{\con{\gamma}}. \nonumber
\end{align}
\end{lem}

  \begin{proof}

 We start by differentiating in $t$ the relations  
 $$
 \theta^\alpha_{(t)}(Z_{\beta(t)}) = \delta^\alpha_\beta; \qquad \theta^\alpha_{(t)}(Z_{\con{\beta}(t)}) = \theta^\alpha_{(t)}(T) = 0,  
 $$
 to get 
 $$
 \dot{\theta}^\alpha(Z_\beta) + \theta^\alpha(\dot{Z}_\beta) = 0; \qquad 
 \dot{\theta}^\alpha(Z_{\con{\beta}}) + \theta^\alpha(\dot{Z}_{\con{\beta}}) = 0; \qquad 
 \dot{\theta}^\alpha(T) = 0. 
 $$
 Recall that by Lemma \ref{l:zdot} one has 
 \begin{equation}\label{eq:0}
 	 \dot{\theta}^\alpha = - i E^\alpha_{\con{\beta}} \theta^{\con{\beta}} - F^\alpha_\beta \theta^{\beta}. 
 \end{equation}
%
%
%
%
%
%
%
Differentiate now in $t$  the structure equation 
$$
d \theta^\alpha = \theta^\beta_{(t)} \wedge \omega^\alpha_{\beta(t)} + A^\alpha_{\con{\gamma}(t)} 
\theta \wedge \theta^{\con{\gamma}}_{(t)}
$$ to get 
\begin{align} \nonumber
 d \dot{\theta}^\alpha & =  \dot{\theta}^\beta \wedge \omega^\alpha_\beta + \theta^\beta \wedge \dot{\omega}^\alpha_\beta
+ \dot{A}^\alpha_{\con{\gamma}} \theta \wedge \theta^{\con{\gamma}} + A^\alpha_{\con{\gamma}} \theta \wedge
\dot{\theta}^{\con{\gamma}} \\  \label{eq:cheng-1}
 & =  - i d E^\alpha_{\con{\beta}} \wedge \theta^{\con{\beta}} - i E^\alpha_{\con{\beta}} d \theta^{\con{\beta}} - d F^\alpha_\beta \wedge 
 \theta^\beta - F^\alpha_\beta d \theta^\beta. 
\end{align}
At a given point $p$ we may assume that $\omega^\alpha_\beta = 0$, by Lemma \ref{l:omegazero}. Therefore we obtain at $p$ 
$$
  - i d E^\alpha_{\con{\beta}} \wedge \theta^{\con{\beta}} = - i E^\alpha_{\con{\beta},0} \theta \wedge \theta^{\con{\beta}} 
  	- i E^\alpha_{\con{\beta},\gamma} \theta^\gamma \wedge \theta^{\con{\beta}} 
  	- i E^\alpha_{\con{\beta},\con{\gamma}} \theta^{\con{\gamma}} \wedge \theta^{\con{\beta}}; 
$$
$$
 d \theta^{\con{\beta}} = \theta^{\con{\gamma}} \wedge \omega^{\con{\beta}}_{\con{\gamma}} 
 + A^{\con{\beta}}_\gamma \theta \wedge \theta^\gamma = A^{\con{\beta}}_\gamma \theta \wedge \theta^\gamma; 
$$
$$
  - d F^\alpha_\beta \wedge 
  \theta^\beta - F^\alpha_\beta d \theta^\beta = - F^\alpha_{\beta,0} \theta \wedge \theta^\beta 
  - F^\alpha_{\beta,\gamma} \theta^\gamma \wedge \theta^\beta 
  - F^\alpha_{\beta,\con{\gamma}} \theta^{\con{\gamma}} \wedge \theta^\beta
  - F^\alpha_\beta A^\beta_{\con{\gamma}} \theta \wedge \theta^{\con{\gamma}}.   
$$
Comparing the coefficients of $\theta \wedge \theta^{\con{\gamma}}$ in \eqref{eq:cheng-1}, we deduce the 
first assertion.

 Write next 
 \begin{equation} \label{eq:do}
 	\dot{\omega}^\alpha_\beta = x^\alpha_\beta \theta + y^\alpha_{\beta \gamma} \theta^\gamma + 
 	y^\alpha_{\beta \con{\gamma}} \theta^{\con{\gamma}}. 
 \end{equation}
 From the relation $\omega^\beta_\alpha + \omega^{\con{\alpha}}_{\con{\beta}}  = 0$, which 
 implies $\dot{\omega}^\beta_\alpha + \dot{\omega}^{\con{\alpha}}_{\con{\beta}}  = 0$,  we get the system
 $$
 \begin{cases}
 	x^\alpha_\beta + x^{\con{\beta}}_{\con{\alpha}} = 0; \\ 
 	y^\alpha_{\beta \gamma} = - \overline{y^\beta_{\alpha \con{\gamma}}} := - y^{\con{\beta}}_{\con{\alpha}\gamma}. 
 \end{cases}
 $$
  Substituting \eqref{eq:0}, \eqref{eq:d-tor} and \eqref{eq:do}  into \eqref{eq:cheng-1} we obtain  
  \begin{equation}\label{eq:xy}
  	\begin{cases}
  		x^\alpha_\beta = i E^\alpha_{\con{\gamma}} A^{\con{\gamma}}_\beta + i A^\alpha_{\con{\gamma}} E^{\con{\gamma}}_\beta
  		+ F^\alpha_{\beta,0}; \\ 
  		y^\alpha_{\beta \con{\gamma}} = - i E^\alpha_{\con{\gamma},\beta} + F^\alpha_{\beta,\con{\gamma}},
  	\end{cases}
  \end{equation}
 which implies in particular  $y^\alpha_{\beta {\gamma}} = - i E^{\con{\beta}}_{\gamma,\con{\alpha}} 
 - F^{\con{\beta}}_{\con{\alpha},\gamma}$. We also get the relations 
 $$
   (y^\alpha_{\beta \gamma} - F^\alpha_{\beta, \gamma}) \theta^\beta \wedge \theta^\gamma = 0; \qquad 
   - i E^\alpha_{\con{\beta}, \con{\gamma}} \theta^{\con{\gamma}} \wedge \theta^{\con{\beta}} = 0,
 $$
giving the following constraints on the deformation tensors 
   $$
 \begin{cases}
 	y^\alpha_{\beta \gamma} - F^\alpha_{\beta, \gamma} = y^\alpha_{\gamma \beta} - F^\alpha_{\gamma, \beta}; \\ 
 	E^\alpha_{\con{\beta}, \con{\gamma}} = E^\alpha_{\con{\gamma}, \con{\beta}}. 
 \end{cases}
 $$
 In this way, we obtain  the second assertion as well. 
 \end{proof}
 
 We can now compute the derivative of the curvature tensor with respect to $t$, together with its contractions.

 \begin{pro} \label{p:var-curv}
 	For the curvature tensor, the Ricci tensor and the Webster curvature we have the following 
 	variation formulas 
 	\begin{align} \label{eq:var-curv}
 		\dot{R}^\alpha_{\beta \rho \con{\sigma}} & =  y^\alpha_{\beta \con{\sigma} \rho} - y^\alpha_{\beta \rho \con{\sigma}} 
 		+ i x^\alpha_\beta \delta_{\rho \con{\sigma}} + R^\alpha_{\beta l \con{\sigma}} F^l_\rho + R^\alpha_{\beta \rho \con{l}} F^{\con{l}}_{\con{\sigma}}
 		\\   & + A_{\con{\alpha} \con{\sigma}} E^{\con{\beta}}_\rho + A_{\beta \rho} E^\alpha_{\con{\sigma}} - 
 		A_{\con{\alpha} \con{\gamma}} E^{\con{\gamma}}_\rho \delta_{\beta \con{\sigma}} 
 		- A_{\beta \gamma} E^\gamma_{\con{\sigma}} \delta_{\alpha \rho}; \nonumber
 	\end{align}
 	\begin{align} \label{eq:var-Ricci}
 		\dot{R}_{\rho \con{\sigma}(t)} = i E^{\con{\gamma}}_{\rho, \con{\gamma} \con{\sigma}} - i E^\gamma_{\con{\sigma}, \gamma \rho} 
 		- (A^{\con{\gamma}}_l E^l_{\con{\gamma}} + A^l_{\con{\gamma}} E^{\con{\gamma}}_l) \delta_{\rho \con{\sigma}} 
 		+ R_{l \con{\sigma}} F^l_\rho + R_{\rho \con{l}} E^{\con{l}}_{\con{\sigma}}, 
 	\end{align}
and 
 	\begin{equation}\label{eq:first-var-R}
 		\dot{W} = \dot{R}_{\alpha \con{\alpha}} = i E^{\con{\gamma}}_{l, \con{\gamma} \con{l}} - i E^\gamma_{\con{l}, \gamma l} 
 		- (A^{\con{\gamma}}_l E^l_{\con{\gamma}} + A^l_{\con{\gamma}} E^{\con{\gamma}}_l) n
 		+ R_{l \con{\gamma}} F^l_\gamma + R_{r \con{\gamma}} F^{\con{\gamma}}_{\con{r}}. 
 	\end{equation}
 \end{pro}

 \begin{proof}

 Differentiate  in $t$ the structure equation (see Section 2)
 \begin{align*}
 	d \omega^\alpha_{\beta(t)} - \omega^\gamma_{\beta(t)} \wedge \omega^\alpha_{\gamma(t)} & = 
 	R^\alpha_{\beta \rho \con{\sigma}(t)} \theta^\rho_{(t)} \wedge \theta^{\con{\sigma}}_{(t)} 
 	+ W^\alpha_{\beta \rho(t)} \theta^{\rho}_{(t)} \wedge \theta - W^\alpha_{\beta \con{\rho}} 
 	\theta^{\con{\rho}}_{(t)} \wedge \theta \\ & + i \theta_{\beta (t)} \wedge \tau^\alpha_{(t)} 
 	- i \tau_{\beta (t)} \wedge \theta^\alpha_{(t)} 
 \end{align*}
  where, we recall,  
  $$
    \tau_\beta = A_{\beta \gamma} \theta^\gamma; \qquad 
    \tau^\alpha = A^\alpha_{\con{\gamma}} \theta^{\con{\gamma}}, 
  $$
 and $A^\alpha_{\con{\gamma}} = A_{\con{\alpha} \con{\gamma}}$ since $h_{\alpha \con{\beta}} = \delta_{\alpha}^{\beta}$. 
 We then deduce 
 \begin{align} \label{eq:dot-str} \nonumber
 	d \dot{\omega}^\alpha_{\beta} - \dot{\omega}^\gamma_{\beta} \wedge \omega^\alpha_{\gamma} 
 	- \omega^\gamma_\beta \wedge \dot{\omega}^\alpha_\gamma & =  
 	\dot{R}^\alpha_{\beta \rho \con{\sigma}} \theta^\rho \wedge \theta^{\con{\sigma}} 
 	+ R^\alpha_{\beta \rho \con{\sigma}} (\dot{\theta}^\rho \wedge \theta^{\con{\sigma}} + \theta^\rho 
 	\wedge \dot{\theta}^{\bar{\sigma}}) \\ & +  i \dot{\theta}^{\con{\beta}} \wedge A_{\con{\alpha} \con{\gamma}} 
 	\theta^{\con{\gamma}} + i \dot{A}_{\con{\alpha} \con{\gamma}} \theta^{\con{\beta}} 
 	\wedge \theta^{\con{\gamma}} + i A_{\con{\alpha} \con{\gamma}} \theta^{\con{\beta}} \wedge 
 	\dot{\theta}^{\con{\gamma}} \\ & -  i A_{\beta \gamma} \dot{\theta}^\gamma \wedge \theta^\alpha 
 	- i \dot{A}_{\beta \gamma} \theta^\gamma \wedge \theta^\alpha 
 	- i A_{\beta \gamma} \theta^\gamma \wedge \dot{\theta}^\alpha \nonumber \\ 
 	& \hbox{mod. } \quad \theta^\alpha \wedge \theta  \quad \hbox{ and } \quad \theta^{\con{\alpha}} \wedge \theta. \nonumber
 \end{align}
 Writing 
 $$
  d \dot{\omega}^\alpha_\beta = x^\alpha_\beta d \theta + y^\alpha_{\beta \gamma \con{l}} \theta^{\con{l}} 
  \wedge \theta^\gamma + y^\alpha_{\beta \gamma \con{l}} \theta^{\con{l}} \wedge \theta^\gamma 
  + y^\alpha_{\beta \con{\gamma} l} \theta^l \wedge \theta^{\con{\gamma}}, 
 $$
  keeping only terms of the type $\theta^\rho \wedge \theta^{\con{\sigma}}$ and 
 using $\omega^\alpha_\gamma(p) = 0$, we get 
 \begin{align} \label{eq:rr}
 	\dot{R}^\alpha_{\beta \rho \con{\sigma}} & =  y^\alpha_{\beta \con{\sigma} \rho} - y^\alpha_{\beta \rho \con{\sigma}} 
 	+ i x^\alpha_\beta \delta_{\rho \con{\sigma}} + R^\alpha_{\beta l \con{\sigma}} F^l_\rho + R^\alpha_{\beta \rho \con{l}} F^{\con{l}}_{\con{\sigma}}
 \\   & + A_{\con{\alpha} \con{\sigma}} E^{\con{\beta}}_\rho + A_{\beta \rho} E^\alpha_{\con{\sigma}} - 
 A_{\con{\alpha} \con{\gamma}} E^{\con{\gamma}}_\rho \delta_{\beta \con{\sigma}} 
 - A_{\beta \gamma} E^\gamma_{\con{\sigma}} \delta_{\alpha \rho}.  \nonumber
 \end{align}
  Recall that, from \eqref{eq:xy}, 
  $$
  y^\alpha_{\beta \con{\sigma}} = - i E^\alpha_{\con{\sigma},\beta} + F^\alpha_{\beta,\con{\sigma}} 
  \quad \Rightarrow \quad y^\alpha_{\beta \con{\sigma} \rho} = - i E^\alpha_{\con{\sigma}, \beta \rho} 
  + F^\alpha_{\beta, \con{\sigma} \rho}; 
  $$
  $$
  y^\alpha_{\con{\beta} {\gamma}} = - i E^{\con{\beta}}_{\rho, \con{\alpha}} - F^{\con{\beta}}_{\con{\alpha}, \rho} 
  \quad \Rightarrow \quad y^\alpha_{\beta \rho \con{\sigma}} = - i E^{\con{\beta}}_{\rho, \con{\alpha}\con{\sigma}} 
  - F^{\con{\beta}}_{\con{\alpha}, \rho \con{\sigma}}
  $$
  and that $x^\alpha_\beta = i(A^{\con{\gamma}}_\beta E^\alpha_{\con{\gamma}} + A^\alpha_{\con{\gamma}} + 
  A^\alpha_{\con{\gamma}} E^{\con{\gamma}}_\beta) + F^\alpha_{\beta,0}$. These last formulas, together with  
  \eqref{eq:rr}, yield \eqref{eq:var-curv}.

 Contracting then \eqref{eq:var-curv}, for the Ricci tensor $R_{\rho \con{\sigma}(t)} := R^\gamma_{\gamma \rho \con{\sigma}(t)}$ we  obtain, after some cancellation that uses $A_{\alpha \beta} = A_{\beta \alpha}$  
 \begin{align}
 	\dot{R}_{\rho \con{\sigma}(t)} = i E^{\con{\gamma}}_{\rho, \con{\gamma} \con{\sigma}} - i E^\gamma_{\con{\sigma}, \gamma \rho} 
 	- (A^{\con{\gamma}}_l E^l_{\con{\gamma}} + A^l_{\con{\gamma}} E^{\con{\gamma}}_l) \delta_{\rho \con{\sigma}} 
 	+ R_{l \con{\sigma}} F^l_\rho + R_{\rho \con{l}} E^{\con{l}}_{\con{\sigma}}. 
 \end{align}
 We then obtain for the Webster curvature $W_{(t)} = R_{\alpha \con{\alpha}}$ with a further contraction (recall that $h_{\alpha \con{\beta}} = \delta_{\alpha \beta}$) 
 $$
 \dot{W} = \dot{R}_{\alpha \con{\alpha}} = i E^{\con{\gamma}}_{l, \con{\gamma} \con{l}} - i E^\gamma_{\con{l}, \gamma l} 
 - (A^{\con{\gamma}}_l E^l_{\con{\gamma}} + A^l_{\con{\gamma}} E^{\con{\gamma}}_l) n
 + R_{l \con{\gamma}} F^l_r + R_{{\gamma} \con{l}} F^{\con{l}}_{\con{\gamma}},
 $$
 where we used $F^{\con{l}}_{\con{r}} = \overline{F^l_\gamma} = F^l_\gamma$.  
  \end{proof}

 \

 We can now pass to the calculation of the second derivative of the Webster curvature with respect to $t$.

 \begin{pro}
For the second variation of $W = W_{(t)}$ along the deformation $J_{(t)}$,  we have the following formula at $t = 0$ 
 \begin{align} \label{eq:ddotRR}
 	\ddot{W} & =  i \dot{E}^{\con{\gamma}}_{l, \con{\gamma} \con{l}} 
 	- A^{\con{\gamma}}_l \dot{E}^l_{\con{\gamma}}  n
 	+ R_{l \con{\gamma}} \dot{F}^l_\gamma  - n \dot{A}^{\con{\gamma}}_l E^l_{\con{\gamma}}   \\ \nonumber
 	& - E^l_{\con{\rho}} E^{\con{\gamma}}_{\rho, \con{\gamma}l} 
 	- E^l_{\con{\gamma}} E^{\con{\gamma}}_{\rho, l \con{\rho}} 
 	- E^{\con{\gamma}}_l E^l_{\con{\gamma},\rho \con{\rho}} - E^{\con{l}}_\rho 
 	E^l_{\con{\gamma},\gamma \con{\rho}} - E^l_{\con{\gamma}, \con{\rho}} E^{\con{\gamma}}_{\rho, l} \\ \nonumber
 	& - E^{\con{\gamma}}_{l,\con{\rho}} E^l_{\con{\gamma},\rho} - E^{\con{l}}_{\rho, \con{\rho}} E^l_{\con{\gamma},\gamma} 
 	- E^l_{\con{\rho},\rho} E^{\con{\gamma}}_{l,\con{\gamma}} + \text{\emph{conj.}}. 
 \end{align}
 \end{pro}

 \begin{proof}
 	From formula \eqref{eq:first-var-R}, we see that the contribution in the second variation 
 	from $\dot{E}$ and $\dot{F}$ is given by 
 	$$
 	 - (A^{\con{\gamma}}_l\dot{ E}^l_{\con{\gamma}} + A^l_{\con{\gamma}} \dot{E}^{\con{\gamma}}_l) n
 	 + R_{l \con{\gamma}} \dot{F}^l_\gamma + R_{r \con{\gamma}} \dot{F}^{\con{\gamma}}_{\con{r}}, 
 	$$
 	giving the second and third term in the right-hand side of \eqref{eq:ddotRR}, plus their conjugates. 
 	To compute the remaining terms, we can therefore assume that $\dot{ E} = 0$ and $\dot{F} = 0$. 
 	 
 	We will need first some preliminary calculation: recall that 
 	$$
 	 E^{\con{\gamma}}_{\alpha,\beta} = Z_\beta(E^{\con{\gamma}}_\alpha) 
 	 - \omega^l_\alpha(Z_\beta) E^{\con{\gamma}}_l  + \omega^{\con{\gamma}}_{\con{l}}(Z_\beta) E^{\con{l}}_\alpha. 
 	$$
 	Taking the $t$-derivative and using that $\dot{E}  = 0$ at $t = 0$ we get
 	$$
 	 0 = \dot{Z}_\beta(E^{\con{\gamma}}_\alpha) - \left[ \dot{\omega}^l_\alpha(Z_\beta) + \omega^l_\alpha(\dot{Z}_\beta) 
 	 \right] E^{\con{\gamma}}_l + \left[ \dot{\omega}^{\con{\gamma}}_{\con{l}}(Z_\beta) + \omega^{\con{\gamma}}_{\con{l}}
 	 (\dot{Z}_\beta) \right] E^{\con{l}}_\alpha. 
 	$$
 	Recalling that at $t = 0$ we can take $F = 0$, we have  
 	$$
 	  \dot{\omega}^l_\alpha = i (A^l_{\con{\gamma}} E^{\con{\gamma}}_\alpha + E^l_{\con{\gamma}} A^{\con{\gamma}}_\alpha) \theta 
 	  - i E^l_{\con{\gamma}, \alpha} \theta^{\con{\gamma}} - i E^{\con{\alpha}}_{\gamma, \con{l}} \theta^\gamma, 
 	$$
 	which after some calculation implies 
	$$
 	(E^{\con{\gamma}}_{\alpha, \beta})^{\cdot} = 
 	 - i E^{\con{l}}_\alpha E^{\con{\gamma}}_{\beta, \con{l}}
	+ i E^{\con{\beta}}_{\alpha, \con{l}} E^{\con{\gamma}}_l + i E^{\con{l}}_\beta E^{\con{\gamma}}_{\alpha, \con{l}}.
 	$$
 	In a similar manner, from the formula 
 	$$
 	 E^{\con{\gamma}}_{\alpha, \con{\beta}} = Z_{\con{\beta}}(E^{\con{\gamma}}_\alpha) - \omega^l_\alpha(Z_{\con{\beta}}) 
 	 E^{\con{\gamma}}_l + \omega^{\con{\gamma}}_{\con{l}}(Z_{\con{\beta}}) E^{\con{l}}_\alpha
 	$$
 	one finds, for $t = 0$
 	$$
    (E^{\con{\gamma}}_{\alpha, \con{\beta}})^{\cdot} = i E^\rho_{\con{\beta}} E^{\con{\gamma}}_{\alpha, \rho} 
    + i E^{\con{\gamma}}_l E^l_{\con{\beta}, \alpha} + i E^{\con{l}}_\alpha E^l_{\con{\beta}, \gamma}.
 	$$

 	We analyze the  terms with second-order covariant derivatives. Notice that 
 	$$
 	E^{\con{\gamma}}_{\alpha, \con{\beta} \con{\rho}} = Z_{\con{\rho}}(E^{\con{\gamma}}_{\alpha, \con{\beta}}) 
 	- \omega^l_\alpha(Z_{\con{\rho}}) E^{\con{\gamma}}_{l, \con{\beta}} - \omega^{\con{l}}_{\con{\beta}} (Z_{\con{\rho}}) 
 	E^{\con{\gamma}}_{\alpha, \con{l}} + \omega^{\con{\gamma}}_{\con{l}} (Z_{\con{\rho}}) E^{\con{l}}_{\alpha, \con{\beta}}. 
 	$$
 	Since $\omega = 0$ at $t = 0$ and at a given point $p$, this implies 
 	\begin{align*}
 			(E^{\con{\gamma}}_{\alpha, \con{\beta} \con{\rho}})^{\cdot} & = \dot{Z}_{\con{\rho}}(E^{\con{\gamma}}_{\alpha, \con{\beta}}) 
 			+ Z_{\con{\rho}}( \dot{E}^{\con{\gamma}}_{\alpha, \con{\beta}}) 
 		\\ & - \dot{\omega}^l_\alpha(Z_{\con{\rho}}) E^{\con{\gamma}}_{l, \con{\beta}} 
 		- \dot{\omega}^{\con{l}}_{\con{\beta}} (Z_{\con{\rho}}) 
 		E^{\con{\gamma}}_{\alpha, \con{l}} 
 		+ \dot{\omega}^{\con{\gamma}}_{\con{l}} (Z_{\con{\rho}}) E^{\con{l}}_{\alpha, \con{\beta}}.  
 	\end{align*}
 	After some straightforward calculation, one then finds  
 	\begin{align*}
 		& (E^{\con{\gamma}}_{\alpha, \con{\beta} \con{\rho}})^{\cdot} = i E^l_{\con{\rho}} E^{\con{\gamma}}_{\alpha, \con{\beta}l} 
 		+ i E^l_{\con{\rho}, \alpha} E^{\con{\gamma}}_{l, \con{\beta}} + i E^l_{\con{\rho}, \gamma} E^{\con{l}}_{\alpha, \con{\beta}} 
 		- i E^\beta_{\con{\rho},l} E^{\con{\gamma}}_{\alpha, \con{l}} 
 		\\ & + i \left( E^l_{\con{\beta},\con{\rho}} E^{\con{\gamma}}_{\alpha, l} + E^l_{\con{\beta}} E^{\con{\gamma}}_{\alpha, l \con{\rho}}
 		+ E^{\con{\gamma}}_{l,\con{\rho}} E^l_{\con{\beta}, \alpha} + E^{\con{\gamma}}_l E^l_{\con{\beta}, \alpha \con{\rho}} 
 		+ E^{\con{l}}_{\alpha, \con{\rho}} E^l_{\con{\beta}, \gamma} + E^{\con{l}}_\alpha 
 		E^l_{\con{\beta}, \gamma \con{\rho}} \right).  \nonumber
 	\end{align*}
 	In particular, taking the trace we obtain after some cancellation 
 	\begin{align*}
 	 (E^{\con{\gamma}}_{\rho, \con{\beta} \con{\rho}})^{\cdot} & =  i \left( E^l_{\con{\rho}} E^{\con{\gamma}}_{\rho, \con{\gamma} l} 
 	 + E^l_{\con{\gamma}} E^{\con{\gamma}}_{\rho, l \con{\rho}} + E^{\con{\gamma}}_l E^l_{\con{\gamma}, \rho \con{\rho}} 
 	 + E^{\con{l}}_\rho E^l_{\con{\gamma}, \gamma \con{\rho}} \right) 
 	\\ & + i \left( E^l_{\con{\gamma}, \con{\rho}} E^{\con{\gamma}}_{\rho, l} + E^{\con{\gamma}}_{l,\con{\rho}} 
 	E^l_{\con{\gamma},\rho} + E^{\con{l}}_{\rho,\con{\rho}} E^l_{\gamma, \con{\gamma}} 
 	+E^l_{\con{\rho}, \rho} E^{\con{\gamma}}_{l,\con{\gamma}} \right). 
 	\end{align*}

 We have that 
  \begin{align*}
 	\ddot{W} & =  
 i (E^{\con{\gamma}}_{l,\con{\gamma} \con{l}})^{\dot{}} 
- i (E^\gamma_{\con{l},\gamma l})^{\dot{}}
 	 - (\dot{A}^{\con{\gamma}}_l E^l_{\con{\gamma}} + \dot{A}^l_{\con{\gamma}} E^{\con{\gamma}}_l) n 
\\ 	 & +  (\dot{R}_{l \con{\gamma}} F^l_\gamma + R_{l\con{\gamma}} \dot{F}^l_\gamma) + 
(\dot{R}_{\gamma \con{l}} F^{\con{l}}_{\con{\gamma}} + R_{\gamma \con{l}} \dot{F}^{\con{l}}_{\con{\gamma}}). 
  \end{align*}
The second line indeed vanishes since $F^l_\gamma = 0$ at $t = 0$, and since $F^l_l = 0$ implies 
$$
 (R_{l\con{\gamma}} + R_{\gamma \con{l}}) \dot{F}^l_\gamma = \frac{1}{n} (\delta_{l \gamma} + \delta_{\gamma l}) 
 \dot{F}^l_\gamma = \frac{2}{n} W \dot{F}^l_l = 0. 
$$ 
This concludes the proof. 
\end{proof}

 \section{Proof of the  theorems} \label{s:proofs}

 In this section we prove our main results, starting from conformal variations and 
 then passing to variations of the CR structure. 
 
 \

\begin{proof} [Proof of Theorem \ref{t:0}.]
	The constancy of  Webster's curvature is classical and can be obtained as in \cite{JerisonLee}. In fact, 
	from formula \eqref{eq:conf} and an integration by parts we have that 
	\begin{align}
		\tilde{\ci{W}}(J,u^{\frac{4}{{Q-2}}} \theta) & = \frac{\int_M u \left[- \left(2 + \frac 2n \right) \Delta_b u + W_{J,\theta} u \right]  \theta \wedge 
			(d \theta)^n}{\left( \int_{M} u^{\frac{2Q}{Q-2}} \theta \wedge 	(d \theta)^n \right)^{\frac{Q-2}{Q}}} 
		\\ & = \frac{\int_M  \left[ \left(2 + \frac 2n \right) |\nabla_b u|^2 + W_{J,\theta} u^2 \right]  \theta \wedge 
			(d \theta)^n}{\left( \int_{M} u^{\frac{2Q}{Q-2}} \theta \wedge 	(d \theta)^n \right)^{\frac{Q-2}{Q}}}. 
	\end{align}
Given the scaling-invariant character of $\tilde{\ci{W}}$,  when taking conformal variations of the type $(u + tv)^{\frac{2}{n}} \theta$ at $u \equiv 1$ 
we can assume that $\int_M v \, \theta \wedge 
(d \theta)^n = 0$, so we obtain  
$$
  \int_M W_{J,\theta} v \, \theta \wedge 
  (d \theta)^n = 0 \qquad \hbox{ for all } v \hbox{ such that } \int_M v \,  \theta \wedge 
  (d \theta)^n = 0,
$$
implying that $W_{J,\theta}$ is constant. 

	Let us now consider variations in $J$ (leaving then $\theta$ fixed) of $\tilde{\ci{W}}$. We see that 
	the first two terms in the right-hand side of \eqref{eq:first-var-R} vanish after integration, and 
	that the last two terms also vanish pointwise since we can take $F = 0$. Therefore, choosing $E_{\alpha \beta} 
	= A_{\alpha \beta}$, we deduce 
 	the vanishing of the torsion by integration of the formula over $M$. Notice that such variations 
 	are admissible since by the third equation in \eqref{eq:circ} and \eqref{eq:Bianchi} we 
 	have the constraints  given in Lemma \ref{l:E-symm}. 
 	
 	\
 	
 	Let us now check the pseudo-Einstein condition: vanishing of the torsion implies that 
 	the Reeb vector field $T$  corresponding to $\theta$ generates an infinitesimal transverse symmetry, 
 	see e.g. (2.12) and Proposition 2.2 in \cite{Webster}.   By Theorem E in \cite{Lee2}, if we also assume 
 	that $c_1(T_{1,0}(M)) = 0$, then  there exists $u \in C^\infty(M)$ such that $(M,J,e^{2u} \theta)$ is 
 	pseudo-Einstein. In the Appendix it is indeed shown that $u$ can be taken identically zero, 
 	see Proposition \ref{p:p-app}. 
\end{proof}

 Before proving the next theorems, we compute  the second variation of $\tilde{\ci{W}}$ in the 
 conformal directions. The first conformal variation is given by 
 \begin{align}
 	& \frac{d}{dt}|_{t=0} \tilde{\ci{W}}(J,(u+tv)^{\frac{4}{{Q-2}}} \theta)  = 
   2 \frac{\int_{M} \left( b_n \nabla_b u \cdot \nabla_b v + W_{J,\theta} u v\right) \theta \wedge (d \theta)^n}{
   \left( \int_M u^{\frac{2Q}{Q-2}}\theta \wedge (d \theta)^n\right)^{\frac{Q-2}{Q}}}  \\ & - 
2  \frac{\int_M \left( |\nabla_b u|^2 + W_{J,\theta} u^2 \right) \theta \wedge (d \theta)^n 
	}{\left( \int_M u^{\frac{2Q}{Q-2}}\theta \wedge (d \theta)^n\right)^{2\frac{Q-2}{Q}}}
\int_M u^{\frac{Q+2}{Q-2}} v \, \theta \wedge (d \theta)^n. 
  \end{align}
 In this way, if $W_{J,\theta}$ is constant, one sees that criticality occurs when 
\begin{equation}\label{eq:crit-u}
	 Vol_\theta(M)^{\frac{2}{Q}} u^{\frac{4}{Q-2}}  = 1. 
\end{equation}
The second variation at a stationary point is the following  
 \begin{align*} \nonumber
 	& \frac{d^2}{dt^2}|_{t=0} \tilde{\ci{W}}(J,(1+tv)^{\frac{4}{{Q-2}}} \theta)  = 
 	2 \frac{\int_{M} \left( b_n |\nabla_b v|^2 + W_{J,\theta} v^2 \right) \theta \wedge (d \theta)^n}{
 		Vol_{\theta,u}(M)^{\frac{Q-2}{Q}}}  \\ & - 
 	2 \frac{Q+2}{Q-2} \frac{\ W_{J,\theta} Vol_\theta(M) 
 	}{Vol_{\theta,u}(M)^{2\frac{Q-2}{Q}}}
 	\int_M  v^2 \, \theta \wedge (d \theta)^n, 
 \end{align*}
 where $Vol_{\theta,u}(M) = \int_M u^{\frac{2Q}{Q-2}}\theta \wedge (d \theta)^n$. Inserting 
 \eqref{eq:crit-u} into the latter formula we see that the second variation becomes 
 \begin{equation}\label{eq:2nd-var-CR}
 	 \frac{2}{Vol_{\theta,u}(M)^{\frac{Q-2}{Q}}} \int_M \left( b_n |\nabla_b v|^2 
 	- \frac{4}{Q-2} W_{J,\theta} v^2 \right) \theta \wedge (d \theta)^n. 
 \end{equation}
 
 For the standard spheres $(S^{2n+1}, J_0,\theta_0)$, recalling from \cite{JerisonLee} and \cite{Webster} that 
 \begin{equation}\label{eq:const-spheres}
 	   b_n = 2 + \frac{2}{n}; \qquad W_{J_0,\theta_0} = n(n+1), 
 \end{equation}
 we  get the first statements 
 in Theorem \ref{t:1} and Theorem \ref{t:2} (see also \cite{MalchiodiUguzzoni}). 
 
 \begin{rem}
 	More in general, if we are on a pseudo-Einstein manifold with zero torsion other than the 
 	standard sphere, using Theorem 1.1 in \cite{Chiu} and Theorem 3 in \cite{LiWang} for $n = 1$ and 
 	$n > 1$ respectively, 
 	by formula \eqref{eq:2nd-var-CR} the second conformal variation is strictly 
 	positive-definite. 
 \end{rem}
 
 We consider next the second variation of $\tilde{{\ci{W}}}$ on  standard spheres 
 with respect to the deformation of the CR structure.

 \begin{lem} \label{l:0-der}
 	For the standard structure $(S^{2n+1},J_0,\theta_0)$ we have that 
 	\begin{equation}\label{eq:ddot-intR}
 		 	\frac{d^2}{dt^2}|_{t=0} {\tilde{\ci{W}}}(J_{(t)},\theta_0) = -i \, n\int_{S^{2n+1}}{{E_{\alfa}}^{\con{\gamma}}}_{,0}{E_{\con{\gamma}}}^{\alfa} \,
 		 	\theta_0 \wedge (d \theta_0)^n + \textrm{\emph{conj.}},
 	\end{equation}
 where $E = 2 \frac{d}{dt}|_{t=0} J_{(t)}$.  
 \end{lem}

\begin{proof}
	Since $\theta_0$ remains fixed, we just need to integrate $\ddot{W}$ with respect to the 
	volume form $\theta_0 \wedge (d \theta_0)^n$. 
	
	Recalling \eqref{eq:ddotRR}, we first notice that the terms involving $\dot{E}$ and $\dot{F}$ vanish since they correspond to 
	 the first variation of $\tilde{\ci{W}}$ in the direction $\dot{E}$, but we are at a stationary point. 
	 
	 Concerning the quadratic terms in $E$, we observe that after integrating and using Lemma \ref{l:E-symm}, 
	 we obtain cancellation  in \eqref{eq:ddotRR} of the first with the seventh, of the second with the fifth, of the third with 
	 the sixth and of the fourth with the eighth. We are then left with 
	 $$
	   \ddot{\ci{W}}= - n \int_{S^{2n+1}}  (\dot{A}^{\con{\gamma}}_l E^l_{\con{\gamma}} + \dot{A}^l_{\con{\gamma}} E^{\con{\gamma}}_l) 
	   \, \theta_0 \wedge (d \theta_0)^n. 
	 $$
	 Recalling formula \eqref{eq:d-tor} and the fact that we can take $F^\alpha_\beta = 0$ at $t = 0$, 
	 we obtain the desired conclusion. 
\end{proof}

 \
 
 To understand the second variation formula in \eqref{eq:ddot-intR} on the 
 sphere $S^{2n+1}$, instead of employing the above 
 moving frame approach, we use instead a basis of the complexified tangent space that is 
 induced from the ambient space $\C^{n+1}$, introduced in \cite{Geller}. This basis 
 does not consist of linearly independent vectors, but it has the advantage of leading to 
 computable quantities, with  coefficients that are either constant or  that are powers 
 of the $z$- and $\con{z}$-coordinates.

 Let
 \begin{equation}\label{eq:Zjk}
 	Z_{jk}=\con{z}_j\frac{\de}{\de z_k} - \con{z}_k\frac{\de}{\de z_j}, \quad \theta_{jk}= z_jdz_k-z_kdz_j, 
 	\qquad j \neq k. 
 \end{equation}
 We have that
\begin{align}\label{eq:th-zh} \nonumber
	\theta_{\ell m}(Z_{jk}) & = (z_{\ell}dz_m-z_mdz_{\ell})\left(\con{z}_j\frac{\de}{\de z_k} - \con{z}_k\frac{\de}{\de z_j}\right) \\ 
	& = z_{\ell}\con{z}_j\delta_{km}-z_{\ell}\con{z}_k\delta_{jm} -z_m\con{z}_j\delta_{k\ell}+ z_m\con{z}_k\delta_{j\ell}.
\end{align}
 As proven in \cite{Geller}, every form of type $(0,1)$ can be written as
 $$\eta= \sum_{0\le j<k\le n}\eta(\con{Z}_{jk})\con{\theta}_{jk}.$$
 We warn that the coefficients $\eta(\con{Z}_{jk})$ are some functions which may not 
 coincide with $\eta$ applied to $\con{Z}_{jk}$.

 Similarly, any form of type $(1,0)$ can be written as
 $$\eta= \sum_{0\le j<k\le n}\eta(Z_{jk})\theta_{jk},$$
 and any vector field of type $(0,1)$  as
 $$X = \sum_{0\le j<k\le n}\con{\theta}_{jk}(X)\con{Z}_{jk}.$$
  Starting with tensor products of objects of the above form, 
  by linearity a tensor $S$ of type $((0,1);(1,0))$ can be written as
 \begin{equation}\label{eq:dec-S}
 	S=\sum_{j<k,\ell<m}S(\con{\theta}_{jk},Z_{\ell m})\con{Z}_{jk}\otimes\theta_{\ell m}. 
 \end{equation}

 Define the musical flat operator $\sharp^{-1} : \con{\mathcal{H}} \to \Omega^{1,0}(M)$ by 
 $$
 \sharp^{-1}(\con{Z}_{jk}) = i d \theta (\cdot, \con{Z}_{jk}). 
 $$


 \begin{lem}\label{l:cov-der-Zij}
 	We have the following relations
 	$$
 	 \nabla_T Z_{jk} = -iZ_{jk}, \qquad \nabla_{Z_{jk}} Z_{pq} = 0 \quad \hbox{ for all } j, k, p, q; 
 	$$
 	$$
 	\nabla_{Z_{jk}} \con{Z}_{lm} 
 	= (\delta_{kl} \con{z}_j - \delta_{jl} \con{z}_k) (\con{\partial}_m - z_m \sigma^\sharp) 
 	+  (\delta_{jm} \con{z}_k - \delta_{lm} \con{z}_j) (\con{\partial}_l - z_l \sigma^\sharp), 
 	$$
 	 where $\sigma^\sharp = \sum_{\nu=1}^{n+1} \con{z}_\nu \con{\partial}_\nu$. 
 	 Moreover, if $\sharp^{-1}$ is as above, one has that 
 	 $$
 	 \sharp^{-1}(\con{Z}_{jk}) = \theta_{jk} \qquad \hbox{ for all } j,k. 
 	 $$
 \end{lem}

 \begin{proof}
 	 Since the pseudohermitian torsion is zero and $T$ is parallel
 	 \begin{align*} \nonumber
 	 	\nabla_TZ_{jk} & = [T,Z_{jk}]=\frac{i}{2}\sum\left[z_{\alfa}\frac{\de}{\de z_{\alfa}}-\con{z}_{\alfa}\frac{\de}{\de \con{z}_{\alfa}},\con{z}_j\frac{\de}{\de z_k} - \con{z}_k\frac{\de}{\de z_j}\right] \\ & = 
 	 	 \frac{i}{2}\left(-\con{z}_j\frac{\de}{\de z_k}+\con{z}_k\frac{\de}{\de z_j} -\con{z}_j\frac{\de}{\de z_k} +\con{z}_k\frac{\de}{\de z_j}\right) \\ & =i\left(-\con{z}_j\frac{\de}{\de z_k}+\con{z}_k\frac{\de}{\de z_j} \right)= -iZ_{jk}, \nonumber
 	 \end{align*}
 	giving the first assertion. The third one is proved in  (2.4) of \cite{Cheng}. 
 	
 	Let us now turn to the second assertion. From Lemma 3.2 (2) in \cite{Tanaka} (in our notation we have an extra factor $1/2$ in front of the contact form) one  has that 
 	$$
 	d \theta (\nabla_X Y, \con{Z}) = X d \theta (Y, \con{Z}) - d \theta (Y, [X,\con{Z}]_{\con{\mathcal{H}}}); 
 	\qquad X, Y \in \mathcal{H}, Z \in \con{\mathcal{H}}, 
 	$$
 	where $[X,\con{Z}]_{\con{\mathcal{H}}}$ stands for the anti-holomorphic projection. From 
 	(the conjugation of) the formula in Lemma 3.2 (i) of \cite{Tanaka} and Lemma \ref{l:cov-der-Zij}
 	one has that 
 	\begin{align*}
 		& [Z_{jk}, \con{Z}_{lm}]_{\con{\mathcal{H}}}  = \nabla_{Z_{jk}} \con{Z}_{lm} \\ & 
 		= (\delta_{kl} \con{z}_j - \delta_{jl} \con{z}_k) (\con{\partial}_m - z_m \sigma^\sharp) 
 		+  (\delta_{jm} \con{z}_k - \delta_{lm} \con{z}_j) (\con{\partial}_l - z_l \sigma^\sharp). 
 	\end{align*}
 	A direct computation 
 	shows, after some cancellation 
 	\begin{align*}
 		d \theta(Z_{pq}, [Z_{jk},\con{Z}_{lm}]_{\con{\mathcal{H}}}) & = 2 i \con{z}_p \left[ \delta_{kl} \con{z}_k \delta_{qm} 
 		+ \delta_{jm} \con{z}_k \delta_{ql} \right]
 		\\ & - 2 i \con{z}_q \left[ \delta_{kl} \con{z}_j \delta_{pm} + \delta_{jm} \con{z}_k \delta_{pl} \right]. 
 	\end{align*}
 	On the other hand, we have that 
 	$$
 	d \theta(Z_{pq}, \con{Z}_{lm})  = 2 i \con{z}_p (z_l \delta_{qm} - z_m \delta_{ql}) - 2 i \con{z}_q (z_l \delta_{pm} 
 	- z_m \delta_{pl}),
 	$$
 	which implies 
 	\begin{align*}
 		& Z_{jk} (d \theta(Z_{pq}, \con{Z}_{lm})) = \con{z}_j \left[ 2 i \con{z}_p (\delta_{kl} \delta_{qm} - \delta_{km} \delta_{ql}) 
 		- 2 i \con{z}_q (\delta_{kl} \delta_{pm} - \delta_{km} \delta_{pl}) \right] \\ & - 
 		\con{z}_k \left[ 2 i \con{z}_p (\delta_{jl} \delta_{qm} - \delta_{jm} \delta_{ql}) - 2 i \con{z}_q 
 		(\delta_{jl} \delta_{pm} - \delta_{jm} \delta_{pl})  \right]. 
 	\end{align*}
 	This means that $Z_{jk} (d \theta(Z_{pq}, \con{Z}_{lm}))  - d \theta(Z_{pq}, [Z_{jk},\con{Z}_{lm}]_{\con{\mathcal{H}}})  = 0$, 
 	and therefore $d\theta(\nabla_{Z_{jk}} Z_{pq}, \con{Z}_{lm}) = 0$ for all $l, m$: hence we obtain 
 	the second assertion too. 
 	
 	To prove the last property, we notice that 
 	$$
 	 i d \theta = \sum_\alpha (d z_\alpha \wedge d \con{z}_\alpha), 
 	$$
 	and therefore 
 	$$
 	  i d \theta(Z_{lm}, \con{Z}_{jk}) = \con{z}_l z_j \delta_{km} - \con{z}_l z_k 
 	  \delta_{jm} - \con{z}_m z_j \delta_{lk} + \con{z}_m z_k \delta_{lj}. 
 	$$
 	From \eqref{eq:th-zh} we then get  
 	$$
 	  i d \theta(Z_{lm}, \con{Z}_{jk}) = \theta_{jk}(Z_{lm}) \qquad \hbox{ for all indices } j, k, l, m, 
 	$$
 	proving also the last assertion. 
 \end{proof}

 
 

We can now prove our second and third main results.

\

\begin{proof} [Proof of Theorem \ref{t:1}]
	The first statement follows from formula \eqref{eq:2nd-var-CR}, recalling \eqref{eq:const-spheres}, 
	so it remains to prove formula \eqref{eq:dd3}. 
	
	Notice that, by Leibnitz's rule 
	\begin{equation}\label{eq:E-0}
		E_{, 0} := \nabla_T (E_1^{\con{1}} \theta^1 \otimes Z_{\con{1}}) = 
		T(E_1^{\con{1}}) \theta^1 \otimes Z_{\con{1}} + E_1^{\con{1}} ((\nabla_T \theta^1) 
		\otimes Z^{\con{1}} + \theta^1 \otimes (\nabla_T Z_{\con{1}})). 
	\end{equation}
	In our notation, comparing \eqref{eq:Z1th1} and \eqref{eq:Zjk}, and recalling 
	Lemma \ref{l:cov-der-Zij} we have that 
	$$
	  \nabla_T Z_{\con{1}} = i Z_{\con{1}}; \qquad \nabla_T \theta^1 = i \theta^1. 
	$$
	This implies  
	$$
	 E_{, 0} = i \left( \frac{m}{2} + 2 \right)  \qquad \hbox{ for } E_1^{\con{1}} \in  \Gamma_m,
	 $$
	 and in turn 
	 $$
	 \ddot{\ci{W}}= \sum_{m \in \mathbf{Z}} (m+4)\int_{S^3}\left|E^{(m)}\right|^2 \theta_0 \wedge d \theta_0, 
	 $$
	 by Lemma \ref{l:0-der}. This gives the desired conclusion. 
\end{proof}

 \

  \begin{proof} [Proof of Theorem \ref{t:2}]
 In the notation of \cite{Bland} and \cite{BlandDuchamp1}, the eigenspace of a 
 deformed structure corresponding to the eigenvalue $i$ is written as 
 $$
   \hat{\mathcal{H}} = \left\{ X - \bar{\phi}(X) \; : \; X \in \mathcal{H}_0 \right\}; 
   \qquad \phi : \mathcal{H} \to \mathcal{H}. 
 $$
 If $\phi$ is infinitesimal, then we have the following 
 relation to the tensor $E$ from Lemma \ref{l:def-E}:  
 $$
  \phi^{\con{\gamma}}_\alpha = i E^{\con{\gamma}}_\alpha. 
 $$
By Theorem 4.1 in \cite{BlandDuchamp1}, $\phi$ is of the 
form 
\begin{equation}\label{eq:BD}
	  \phi = \con{\partial}_b (\con{\partial}_b f)^\sharp + h_\sigma(h), 
\end{equation}
 where $f, h$ are a complex-valued function and a two-form 
 of type $(0,2)$  whose negative Fourier components are zero (with $h$ determined by $f$).  Here 
 $\con{\partial}_b$ denotes the holomorphic differential of $f$ and 
 $\sharp$ the musical isomorphism from $\Omega^{1,0}(S^{2n+1})$ to $T^{0,1}(S^{2n+1})$. 
 In analogy with \eqref{eq:Gamma-m}, the 
 {\em $m$-th Fourier eigenspace} $\Gamma_m$ for a tensor on $S^{2n+1}$ is defined by the action 
 of the flow generated by the vector field $T$. 
 
 Both the operators $\con{\partial}_b$ and $(\con{\partial}_b \cdot)^\sharp$ commute 
 with the Lie derivative by $T$, and it is noticed on page 102 of \cite{BlandDuchamp2} 
 that $h_\sigma$ preserves the Fourier decomposition. Therefore the tensor 
 $\phi$, and hence $E$ as well, only consist of non-negative Fourier modes. 
 
 As for \eqref{eq:dec-S}, let us write 
 \begin{equation}\label{eq:EEE}
 	 E=\sum_{j<k,\ell<s}E(\con{\theta}_{jk},Z_{\ell s})\con{Z}_{jk}\otimes\theta_{\ell s}:  
 \end{equation}
 since both $\con{Z}_{jk}$ and $\theta_{\ell s}$ are invariant under the action 
 of $T$,  we must have that $E(\con{\theta}_{jk},Z_{\ell s})$ also has 
 only non-negative Fourier modes. 
 
Arguing then as for \eqref{eq:E-0} and using the latter formula, we still obtain that 
 $E_{,0} =  i \left( \frac{m}{2} + 2 \right)$ for $E \in  \Gamma_m$, which by 
 Lemma \ref{l:0-der} 
 gives 
 $$
   \ddot{\ci{W}}= n \sum_{m \in {Z}} (m+4)\int_{S^3}\left|E^{(m)}\right|^2 \theta_0 \wedge d \theta_0. 
 $$
  Recalling that $E^{(m)} = 0$ for $m < 0$, we obtain the conclusion.

One comment on the above proof is due, since the coefficients in the expansion 
\eqref{eq:EEE} for $E$ are not uniquely determined. Near each point of 
$S^{2n+1}$ one could choose linearly independent bases of vector fields and forms, 
which would give the asserted property on the Fourier modes of $E(\con{\theta}_{jk},Z_{\ell s})$, 
proving that in any case $E_{,0} =  i \left( \frac{m}{2} + 2 \right)$ for $E \in  \Gamma_m$. 
 \end{proof}

 \section[Appendix by Xiaodong Wang]{Appendix by Xiaodong Wang\footnote{Department of Mathematics,
 		Michigan State University, 
 		619 Red Cedar Road
 		East Lansing, MI 48824 e-mail: xwang@math.msu.edu}}

 It is a well known fact in Kahler geometry that a Kahler metric $\omega$ on a
 closed complex manifold $M$ with constant scalar curvature must be
 Kahler-Einstein if the Kahler class $\left[  \omega\right]  $ is proportional
 to the first Chern class $c_{1}\left(  M\right)  $. In this appendix we discuss a
 CR analogue of this result.
 
 We still follow \cite{Lee2} as our standard reference on CR geometry. Let $M$ be a
 CR manifold of dimension $2n+1$. The first Chern class of the complex vector
 bundle $T^{1,0}M$ will be simply denoted by $c_{1}\left(  M\right)  $.
 Given a pseudohermitian structure $\theta$, we always work with the
 Tanaka-Webster connection $\nabla$, and $2\pi c_{1}\left(  M\right)  $ is
 then represented by the closed $2$-form%
 \begin{equation}
 	\rho_{\theta}=i\left[  R_{\mu\overline{\nu}}\theta^{\mu}\wedge
 	\theta^{\overline{\nu}}+A_{\alpha\gamma,\overline{\alpha}}\theta^{\gamma
 	}\wedge\theta-A_{\overline{\gamma}\overline{\alpha},\alpha}\theta
 	\overline{^{\gamma}}\wedge\theta\right],\label{cform}%
 \end{equation}
 where $R_{\mu\overline{\nu}}$ is the Ricci curvature and $A$ is the torsion of
 the Tanaka-Webster connection. Throughout this section, we always work with a
 local unitary frame $\left\{  Z_{\alpha}:\alpha=1,\cdots,n\right\}  $.
 
 \begin{lem}\label{l:app-1}
 	Suppose $\phi=f_{\mu\overline{\nu}}\theta^{\mu}\wedge\theta^{\overline{\nu}}$
 	is a $\left(  1,1\right)  $-form. Let $\Lambda\left(  \phi\right)  =\sum
 	_{\mu=1}^{n}f_{\mu\overline{\mu}}$ be its trace. Then%
 	\[
 	-d^{\ast}\phi=f_{\alpha\overline{\nu},\overline{\alpha}}\theta^{\overline{\nu
 	}}+f_{\mu\overline{\alpha},\alpha}\theta^{\mu}+i\Lambda\left(
 	\phi\right)  \theta,
 	\]
 	where $d^{\ast}$ the dual of $d$ with respect to the adapted Riemannian metric.
 \end{lem}
 
 \begin{proof}
 	By a standard formula, in terms of the Levi-Civita connection $\widetilde
 	{\nabla}$
 	\[
 	-d^{\ast}\phi=T\lrcorner\widetilde{\nabla}_{T}\phi+Z_{\alpha}\lrcorner
 	\widetilde{\nabla}_{\overline{Z}_{\alpha}}\phi+\overline{Z}_{\alpha}%
 	\lrcorner\widetilde{\nabla}_{Z_{\alpha}}\phi.
 	\]
 	We recall the relationship between the Levi-Civita connection $\widetilde
 	{\nabla}$ and the Tanka-Webster connection, that can be found in \cite{DragomirTomassini}: for $X,Y$ horizontal
 	\begin{align*}
 		\widetilde{\nabla}_{T}T  & =0,\\
 		\widetilde{\nabla}_{X}T  & =AX+\frac{1}{2}JX,\\
 		\widetilde{\nabla}_{X}Y  & =\nabla_{X}Y-\left[  \left\langle AX,Y\right\rangle
 		+\frac{1}{2}d\theta\left(  X,Y\right)  \right]  T.
 	\end{align*}
 	We compute, using the above identities%
 	\begin{align*}
 		T\lrcorner\widetilde{\nabla}_{T}\phi & =f_{\mu\overline{\nu}}\left(
 		\widetilde{\nabla}_{T}\theta^{\mu}\left(  T\right)  \theta^{\overline{\nu}%
 		}-\widetilde{\nabla}_{T}\theta^{\overline{\nu}}\left(  T\right)  \theta^{\mu
 		}\right) \\
 		& =f_{\mu\overline{\nu}}\left(  -\theta^{\mu}\left(  \widetilde{\nabla}%
 		_{T}T\right)  \theta^{\overline{\nu}}+\theta^{\overline{\nu}}\left(
 		\widetilde{\nabla}_{T}T\right)  \theta^{\mu}\right) 
 		=0,
 	\end{align*}%
 	\begin{align*}
 		& Z_{\alpha}\lrcorner\widetilde{\nabla}_{\overline{Z}_{\alpha}}\phi \\ &
 		=\overline{Z}_{\alpha}f_{\alpha\overline{\nu}}\theta^{\overline{\nu}}%
 		+f_{\mu\overline{\nu}}\widetilde{\nabla}_{\overline{Z}_{\alpha}}\theta^{\mu
 		}\left(  Z_{\alpha}\right)  \theta^{\overline{\nu}}-f_{\mu\overline{\nu}%
 		}\widetilde{\nabla}_{\overline{Z}_{\alpha}}\theta^{\overline{\nu}}\left(
 		Z_{\alpha}\right)  \theta^{\mu}+f_{\alpha\overline{\nu}}\widetilde{\nabla
 		}_{\overline{Z}_{\alpha}}\theta^{\overline{\nu}}\\
 		& =\overline{Z}_{\alpha}f_{\alpha\overline{\nu}}\theta^{\overline{\nu}}%
 		-f_{\mu\overline{\nu}}\theta^{\mu}\left(  \widetilde{\nabla}_{\overline
 			{X}_{\alpha}}Z_{\alpha}\right)  \theta^{\overline{\nu}}+f_{\mu\overline{\nu}%
 		}\theta^{\overline{\nu}}\left(  \widetilde{\nabla}_{\overline{Z}_{\alpha}%
 		}Z_{\alpha}\right)  \theta^{\mu}\\
 		& -f_{\alpha\overline{\nu}}\left(  \theta^{\overline{\nu}}\left(
 		\widetilde{\nabla}_{\overline{Z}_{\alpha}}Z_{\beta}\right)  \theta^{\beta
 		}+\theta^{\overline{\nu}}\left(  \widetilde{\nabla}_{\overline{Z}_{\alpha}%
 		}\overline{Z}_{\beta}\right)  \theta^{\overline{\beta}}+\theta^{\overline{\nu
 		}}\left(  \widetilde{\nabla}_{\overline{Z}_{\alpha}}T\right)  \theta\right) \\
 		& =\overline{Z}_{\alpha}f_{\alpha\overline{\nu}}\theta^{\overline{\nu}}%
 		-f_{\mu\overline{\nu}}\theta^{\mu}\left(  \nabla_{\overline{Z}_{\alpha}%
 		}Z_{\alpha}\right)  \theta^{\overline{\nu}}+f_{\mu\overline{\nu}}%
 		\theta^{\overline{\nu}}\left(  \nabla_{\overline{Z}_{\alpha}}Z_{\alpha
 		}\right)  \theta^{\mu}\\
 		& -f_{\alpha\overline{\nu}}\left(  \theta^{\overline{\nu}}\left(
 		\nabla_{\overline{Z}_{\alpha}}Z_{\beta}\right)  \theta^{\beta}+\theta
 		^{\overline{\nu}}\left(  \nabla_{\overline{Z}_{\alpha}}Z_{\overline{\beta}%
 		}\right)  \theta^{\overline{\beta}}+-\frac{i}{2}\delta_{\alpha}^{\nu
 		}\theta\right) \\
 		& =f_{\alpha\overline{\nu},\overline{\alpha}}\theta^{\overline{\nu}}%
 		+\frac{i}{2}\Lambda\left(  \phi\right)  \theta,
 	\end{align*}
 	and similarly,
 	\begin{align*}
 		& \overline{Z}_{\alpha}\lrcorner\widetilde{\nabla}_{Z_{\alpha}}\phi \\ &
 		=Z_{\alpha}f_{\mu\overline{\alpha}}\theta^{\mu}-f_{\mu\overline{\nu}%
 		}\widetilde{\nabla}_{Z_{\alpha}}\theta^{\overline{\nu}}\left(  \overline
 		{X}_{\alpha}\right)  \theta^{\mu}+f_{\mu\overline{\nu}}\widetilde{\nabla
 		}_{Z_{\alpha}}\theta^{\mu}\left(  \overline{Z}_{\alpha}\right)  \theta
 		^{\overline{\nu}}-f_{\mu\overline{\alpha}}\widetilde{\nabla}_{Z_{\alpha}%
 		}\theta^{\mu}\\
 		& =Z_{\alpha}f_{\mu\overline{\alpha}}\theta^{\mu}+f_{\mu\overline{\nu}}%
 		\theta^{\overline{\nu}}\left(  \widetilde{\nabla}_{Z_{\alpha}}\overline
 		{X}_{\alpha}\right)  \theta^{\mu}-f_{\mu\overline{\nu}}\theta^{\mu}\left(
 		\widetilde{\nabla}_{Z_{\alpha}}\overline{Z}_{\alpha}\right)  \theta
 		^{\overline{\nu}}\\
 		& +f_{\mu\overline{\alpha}}\left(  \theta^{\mu}\left(  \widetilde{\nabla
 		}_{Z_{\alpha}}Z_{\beta}\right)  \theta^{\beta}+\theta^{\mu}\left(
 		\widetilde{\nabla}_{Z_{\alpha}}\overline{Z}_{\beta}\right)  \overline{\theta
 		}^{\beta}+\theta^{\mu}\left(  \widetilde{\nabla}_{Z_{\alpha}}T\right)
 		\theta\right) \\
 		& =Z_{\alpha}f_{\mu\overline{\alpha}}\theta^{\mu}+f_{\mu\overline{\nu}}%
 		\theta^{\overline{\nu}}\left(  \nabla_{Z_{\alpha}}\overline{Z}_{\alpha
 		}\right)  \theta^{\mu}-f_{\mu\overline{\nu}}\theta^{\mu}\left(  \nabla
 		_{Z_{\alpha}}\overline{Z}_{\alpha}\right)  \theta^{\overline{\nu}}\\
 		& +f_{\mu\overline{\alpha}}\left(  \theta^{\mu}\left(  \nabla_{Z_{\alpha}%
 		}Z_{\beta}\right)  \theta^{\beta}+\theta^{\mu}\left(  \nabla_{Z_{\alpha}%
 		}\overline{Z}_{\beta}\right)  \overline{\theta}^{\beta}+\frac{i}%
 		{2}\delta_{\alpha}^{\mu}\theta\right) \\
 		& =f_{\mu\overline{\alpha},\alpha}\theta^{\mu}+\frac{i}{2}%
 		\Lambda\left(  \phi\right)  \theta.
 	\end{align*}
 	We remark that in the calculations the torsion $A$ does not appear because it
 	anti-commutes with $J$ and therefore maps a $\left(  1,0\right)  $-vector to a
 	$\left(  0,1\right)  $-vector and vice versa. Combining these results yields%
 	\[
 	-d^{\ast}\phi=f_{\alpha\overline{\nu},\overline{\alpha}}\theta^{\overline{\nu
 	}}+f_{\mu\overline{\alpha},\alpha}\theta^{\mu}+i\Lambda\left(
 	\phi\right)  \theta, 
 	\]
 	concluding the proof. 
 \end{proof}
 
Recall that a pseudohermitian structure $\theta$ is called pseudo-Einstein if
 $R_{\mu\overline{\nu}}=\frac{W}{n}\delta_{\mu\overline{\nu}}$, where $W$ is
 the scalar curvature. When $n=1$, this is always true. From now on we assume
 $n\geq2$. Lee showed in \cite{Lee2}  that a necessary condition for the existence of
 a pseudo-Einstein structure is $c_{1}\left(  M\right)  =0$.
 
 \begin{pro}\label{p:p-app}
 	Suppose $M$ is a closed CR manifold of dimension $2n+1\geq5$ with
 	$c_{1}\left(  M\right)  =0$. If $\theta$ is a pseudohermitian structure
 	with zero torsion and constant scalar curvature, then it is pseudo-Einstein.
 \end{pro}
 
 \begin{proof}
 	Let $\phi=\rho_{\theta}-\frac{W}{n}d\theta$. Since $A=0$, we have%
 	\begin{align*}
 		\phi & =iR_{\mu\overline{\nu}}\theta^{\mu}\wedge\theta^{\overline{\nu
 		}}-\frac{W}{n}d\theta
 		=i\left(  R_{\mu\overline{\nu}}-\frac{W}{n}\delta_{\mu\overline{\nu
 		}}\right)  \theta^{\mu}\wedge\theta^{\overline{\nu}},
 	\end{align*}
 	It is a real $(1,1)$-form with $\Lambda\left(  \phi\right)  =0$. As $A=0$, we
 	have $R_{\alpha\overline{\nu},\alpha}=R_{\mu\overline{\alpha},\alpha}=0$ by
 	the Bianchi identities. Together with $W$ being constant, we see that $d^{\ast}%
 	\phi=0$ by  Lemma \ref{l:app-1}. As $c_{1}\left(  T^{1,0}M\right)  =0$, $\rho_{\theta}$
 	is exact, i.e. there is a real $1$-form $\chi$ s.t. $\rho_{\theta}=d\chi$.
 	Then $\phi=d\widetilde{\chi}$, where $\widetilde{\chi}=\chi-\frac{W}{n}\theta
 	$. It follows%
 	\[
 	\left\Vert \phi\right\Vert ^{2}=\left\langle \phi,d\widetilde{\chi
 	}\right\rangle =\left\langle d^{\ast}\phi,\widetilde{\chi}\right\rangle =0.
 	\]
 	Therefore $\phi=0$, i.e. $\theta$ is pseudo-Einstein.
 \end{proof}

\end{document}